\documentclass{amsart}
\usepackage{latexsym,bbm,amsxtra,amscd,ifthen}
\usepackage{amsfonts}
\usepackage{verbatim}
\usepackage{amsmath}
\usepackage{amsthm}
\usepackage{amssymb}
\usepackage{url}
\usepackage[arrow,matrix]{xy}
\usepackage[colorlinks=true,citecolor=blue,linkcolor=blue]{hyperref}

\theoremstyle{plain}

\newtheorem{theorem}{Theorem}
\newtheorem{lemma}[theorem]{Lemma}
\newtheorem{proposition}[theorem]{Proposition}
\newtheorem{corollary}[theorem]{Corollary}

\numberwithin{theorem}{section}
\numberwithin{equation}{theorem}

\theoremstyle{definition}
\newtheorem{definition}[theorem]{Definition}
\newtheorem{setup}[theorem]{Setup}
\newtheorem{example}[theorem]{Example}
\newtheorem{remark}[theorem]{Remark}

\newtheorem*{question*}{Question}

\DeclareMathOperator{\ch}{char}

\DeclareMathOperator{\RHom}{RHom}

\DeclareMathOperator{\End}{End}
\DeclareMathOperator{\Ext}{Ext}
\DeclareMathOperator{\Tor}{Tor}
\DeclareMathOperator{\Hom}{Hom}

\DeclareMathOperator{\id}{injdim}
\DeclareMathOperator{\GKdim}{GKdim}

\DeclareMathOperator{\gldim}{gldim}
\DeclareMathOperator{\im}{im}
\DeclareMathOperator{\depth}{depth}
\DeclareMathOperator{\Aut}{Aut}

\DeclareMathOperator{\Mod}{Mod}

\DeclareMathOperator{\coker}{coker}
\def\Pty{{\sf p}}
\def\mfr{\mathfrak r}\def\mfb{\mathfrak b}
\def\mfa{\mathfrak a}

\def\kk{\mathbbm{k}}

\def\QMod{\operatorname{QMod}}

\def\ch{\operatorname{char}}

\begin{document}

\title[Local cohomology]
{Local cohomology associated to the radical of a group action on a noetherian algebra}

\author{Ji-Wei He and Yinhuo Zhang}

\address{He: Department of Mathematics,
Hangzhou Normal University,
Hangzhou Zhejiang 310036, China}
\email{jwhe@hznu.edu.cn}

\address{Y. Zhang\newline
\indent Department of mathematics and statistics, University of Hasselt, Universitaire Campus,
3590 Diepenbeek, Belgium} \email{yinhuo.zhang@uhasselt.be}

\begin{abstract} An arbitrary group action on an algebra $R$ results in an ideal $\mfr$ of $R$. This ideal $\mfr$ fits into the classical radical theory, and will be called the radical of the group action. If $R$ is a noetherian algebra with finite GK-dimension and $G$ is a finite group, then the difference between the GK-dimensionsof $R$ and that of $R/\mfr$ is called the pertinency of the group action. We provide some methods to find elements of the radical, which helps to calculate the pertinency of some special group actions. The $\mfr$-adic local cohomology of $R$ is related to the singularities of the invariant subalgebra $R^G$. We establish an equivalence between the quotient category of the invariant $R^G$ and that of the skew group ring $R*G$ through the torsion theory associated to the radical $\mfr$. With the help of  the equivalence, we show that the invariant subalgebra $R^G$ will inherit certain Cohen-Macaulay property from $R$.
\end{abstract}





\maketitle


\setcounter{section}{-1}
\section{Introduction}

Let $R$ be a noetherian algebra with finite GK-dimension, and let $G$ be a finite group acting on $R$. In general, it is not easy to study the structure and the representations of a group action. However, one possible way to understand the group action is through the McKay correspondence, which establishes certain correspondence among the representations of the group $G$, of the skew group algebra $R*G$, and of the invariant subalgebra $R^G$ (see \cite{B} for a good survey of the classical McKay correspondence, or \cite{Mo,CKWZ1,CKWZ2} for noncommutative McKay correspondence). An important step for establishing noncommutative McKay correspondence is the Auslander Theorem in the noncommutative case (cf. \cite{BHZ,BHZ2,GKMW}). To prove certain noncommutative version of the Auslander theorem, Bao-Zhang and the first named author introduced an invariant $\Pty(R,G)$, called the pertinency of the group action. In some sense, the pertinency $\Pty(R,G)$ determines whether the Auslander theorem holds or not (cf. \cite{BHZ}).

The invariant subalgebra $R^G$ is not regular in general. For instance, if $R$ is Artin-Schelter regular and the homological determinant of $G$ is trivial, then $R^G$ is Artin-Schelter Gorenstein \cite{JZ}. So, studying the representations of the noncommutative singularities of $R^G$ is necessary. When $R^G$ has isolated singularities, the representations of the singularities are well understood (cf. \cite{MU1,MU2,U}). Iyama-Wemyss developed a wonderful theory on the representations of non-isolated singularities when the algebra under consideration is commutative (cf. \cite{IW}). The representations of noncommutative non-isolated singularities remain less understood. The main purpose of this paper is to understand the properties of the invariant subalgebra $R^G$ when it has non-isolated singularities.

The main tool we will use in this paper is the {\it radical} $\mfr(R,G)$ of the group action. Let $B=R*G$ be the skew group ring, and let $e=\frac{1}{|G|}\sum_{g\in G}g\in \kk G$, the integral of $\kk G$. The radical $\mfr(R,G)$ is in fact the the intersection $R\cap BeB$. The radical can be introduced even if the group $G$ is not finite. We introduce the concept of two {\it pertinent sequences } of elements of $R$ (cf. Section \ref{sec1}), through which we define the radical $\mfr(R,G)$ for any group action. The radical $\mfr(R,G)$ is an ideal of $R$. We show in Section \ref{sec2} that $\mfr(R,G)$ fits into the the classical radical theory which justifies the name. It is usually difficult to compute $R\cap BeB$ directly, while there are several natural ways to find pertinent sequences  which lead to determine elements in the radical.

The pertinency of a group action $\Pty(R,G)$ introduced in \cite{BHZ} is equal to $\GKdim(R)-\GKdim(R/\mfr(R,G))$ (cf. Section \ref{sec3}). The condition $\Pty(R,G)=1$ is related to Shephard-Todd-Chevalley's Theorem (see Propositions \ref{prop2-3} and \ref{prop2-4}). As it was shown in \cite{BHZ} that the condition $\Pty(R,G)\ge2$ is equivalent to the Auslander Theorem when $R$ is GK-Cohen-Macaulay. We recover some known results in Section \ref{set4} by using  pertinent sequences .

Set $A:=R^G$, $\mfr=\mfr(R,G)$ and $\mfa=A\cap \mfr$. We mainly focus on the local cohomology of $A$-modules associated to the ideal $\mfa$. In Section \ref{sec5}, we established some isomorphisms of the local cohomology (cf. Theorem \ref{thm-locdual}). As one of our main results, we have the following theorem (for the terminology, see Section \ref{sec7}).

\begin{theorem}[Theorem \ref{thm-gcm}] \label{thm0.1} Let $G$ be a finite group, and let $R$ be a noetherian left $G$-module algebra with finite GK-dimension. Assume that $R$ satisfies Setup \ref{set}. If $R$ is $G$-Cohen-Macaulay, then $A$ is $\mfa$-Cohen-Macaulay of dimension $\Pty(R,G)$.
\end{theorem}
As a corollary, we have the following equivalence of local cohomology, where $R^i\Gamma_\mfa$ is the $i$th local cohomology associated to the ideal $\mfa$.

\begin{corollary}[Corollary \ref{cor-cm}] Let $R$ be a $G$-Cohen-Macaulay algebra. Assume that $R$ satisfies Setup \ref{set}.  If $R$ has finite global dimension, then we have
$$R^i\Gamma_\mfa(M)\cong \Tor_{p-i}^A(M,D)$$ for $M\in\Mod A$ and $i\ge0$, where $p=\Pty(R,G)$ and $D=R^p\Gamma_\mfa(A)$.
\end{corollary}

We will list some examples in Sections \ref{sec-quot}, \ref{sec6} and \ref{sec7} which satisfy Setup \ref{set}.  In case the algebra $R$ is commutative, then most conditions in Setup \ref{set} are automatically satisfied.

Theorem \ref{thm0.1} is obtained via an equivalence of quotient categories associated to the ideals $\mfa$ and $\mfr$. Let $\Mod A$ be the category of right $A$-modules, and let $\Tor_\mfa A$ be the full subcategory of $\Mod A$ consisting of $\mfa$-torsion modules. Since $R$ is noetherian and $G$ is finite, $A$ is noetherian as well. Thus $\Tor_\mfa A$ is a Serre subcategory of $\Mod A$. Hence we obtain an abelian category $\QMod_\mfa A:=\frac{\Mod A}{\Tor_\mfa A}$. Similarly, let $\mfb=\mfr\otimes\kk G\subseteq B=R*G$, we have $\QMod_\mfb B:=\frac{\Mod B}{\Tor_\mfb B}$. The following equivalence of quotient categories holds (for the terminology, see Section \ref{sec-quot}).

\begin{theorem}  [Theorem 6.4] Let $R$ and $G$ be as above.  Assume that $\mfa$ and $\mfr$ are cofinal, and $\mfa$ has the right AR-property.
\begin{itemize}
  \item [(i)] The functor $\Hom_A(R,-):\Mod A\longrightarrow \Mod B$ induces a functor $$\Hom_\mathcal{A}(\mathcal{R},-):\QMod_\mfa A\longrightarrow \QMod_\mfb B.$$
  \item [(ii)] The functors $-\otimes_\mathcal{B}\mathcal{R}$ and $\Hom_\mathcal{A}(\mathcal{R},-)$ are quasi-inverse to each other.
\end{itemize}
\end{theorem}

If $R$ is a commutative algebra, then the cofinal condition in the above theorem is indeed necessary (cf. Theorem \ref{thm-commeq}). The extension groups of objects in $\QMod_\mfb B$ can be obtained through the extension groups of the corresponding objects in $\QMod_\mfr R$ (cf. Proposition \ref{prop-qext}), and thus we establish connections among extension groups of objects in $\QMod_\mfa A$, $\QMod_\mfr R$ and $\QMod_\mfb B$ respectively.

Throughout the paper, $\kk$ is a field. All the algebras and modules are assumed to be $\kk$-vector space. Unadorned $\otimes$ means $\otimes_{\kk}$. All the finite groups considered in the paper are nontrivial.

\section{Radicals of group actions} \label{sec1}

Let $R$ be a noetherian algebra. Let $G$ be a finite group which acts on $R$ by automorphisms so that $R$ is a left $G$-module algebra. The $G$-action determines an ideal of $R$.

\begin{definition} \label{def-pelmt} We say that two sequences $(a_1,\dots,a_n)$ and $(b_1,\dots,b_n)$ of elements of $R$ are {\it pertinent under the $G$-action}, if $\sum_{i=1}^na_i(g\cdot b_i)=0$ for all $1\neq g\in G$. We write $(a_1,\dots,a_n)\overset{G}\sim(b_1,\dots,b_n)$ for pertinent sequences.
\end{definition}

\begin{example}\label{example1} (i) Let $R=\kk[x,y]$. Let $\sigma$ be the automorphism of $A$ defined by $\sigma(x)=y$ and $\sigma(y)=x$. Let $G=\{1,g\}$. Then the  sequences $(x,y)\overset{G}\sim(x,-y)$.

(ii) Let $R=\kk_{-1}[x,y]$ be the skew symmetric algebra. Let $\sigma$ be the automorphism of $A$  defined by $\sigma(x)=y$ and $\sigma(y)=x$. Let $G=\{1,g\}$. Then the sequences $(x,y)\overset{g}\sim(x,y)$.
\end{example}

For the convenience of the narratives, we conventionally write $(a_1,\dots,a_n)\vee (a'_1,\dots,a'_m)$ for the sequence $(a_1,\dots,a_n,a'_1,\dots,a'_m)$.
We have the following basic properties of pertinent sequences.

\begin{proposition}\label{prop-basic} Assume $(a_1,\dots,a_n)\overset{G}\sim(b_1,\dots,b_n)$.  The following hold:
\begin{itemize}
  \item [(i)] $(h\cdot a_1,\dots,h\cdot a_n)\overset{G}\sim(h\cdot b_1,\dots,h\cdot b_n)$ for all $h\in G$;
  \item [(ii)] $(a a_1,\dots,aa_n)\overset{G}\sim(b_1b,\dots, b_nb)$ for all $a,b\in A$;
  \item [(iii)] if $(a'_1,\dots,a'_m)\overset{G}\sim(b'_1,\dots,b'_m)$, then $(a_1,\dots,a_n)\vee (a'_1,\dots,a'_m)\overset{G}\sim(b_1,\dots ,b_n)\vee (b'_1,\dots,b'_m)$.
\end{itemize}
\end{proposition}
\begin{proof} Straightforward.
\end{proof}

Given a $G$-action $\theta:G\to \Aut(R)$, we define $$\mfr(R,G)=\left\{\sum_{i=1}^na_ib_i|(a_1,\dots,a_n)\overset{G}\sim(b_1,\dots,b_n),n\ge1\right\}.$$ By Proposition \ref{prop-basic}(ii), $\mfr(R,G)$ is closed under addition, and by Proposition \ref{prop-basic}(ii) $\mfr(R,G)$ is an ideal of $R$.

\begin{definition}\label{def-pideal} We call $\mfr(R,G)$ {\it the radical of $R$ under the $G$-action}, or simply, {\it the $G$-radical of $R$}, and we call the quotient algebra $\mathfrak{R}(R,G):=R/\mfr(R,G)$ the {\it pertinency algebra} of $R$ under the $G$-action.
\end{definition}

\begin{remark} (i) We will show in the next section that the ideal $\mfr(R,G)$ fits into the classical radical theory of rings (cf. \cite{Sz}).

(ii) The $G$-radical $\mfr(R,G)$ depends mainly on the group action $\theta:G\to \Aut(R)$.
\end{remark}

\begin{example} (i) Let $R=\kk[x,y]$. Let $\sigma$ be the automorphism of $A$ defined by $\sigma(x)=y$ and $\sigma(y)=x$. Let $G=\{1,\sigma\}$. Note that $(x,1)$ and $(1,-y)$ are pertinent sequences. Hence $x-y\in \mfr(R,G)$. We claim that $x^n$ ($n\ge2$) is not in $\mfr(R,G)$. Indeed, if $x^n\in \mfr(R,G)$, then there are pertinent sequences $(a_1,\dots,a_m)$ and $(b_1,\dots,b_m)$ such that $x^n=\sum_{i=1}^ma_ib_i$. Note that $R$ has a basis $\{x^sy^t|s,t\ge0\}$. By a simple comparison of powers of $y$ in the sum $\sum_{i=1}^ma_ib_i$, we may assume that $a_i=\sum_jk_{i_j}x^{s_{i_j}}$ and $b_i=r_ix^{t_i}$, where $k_{i_j},r_i\in\kk$ and $t_1,\dots t_m$ are different numbers. Since $(a_1,\dots,a_m)$ and $(b_1,\dots,b_m)$ are pertinent sequences, we have $0=\sum_{i=1}^ma_i\sigma(b_i)=\sum_{i=1}^m\sum_jk_{i_j}r_ix^{s_{i_j}}y^{t_i}$. Then $\sum_jk_{i_j}r_ix^{s_{i_j}}=0$ for all $i$. Therefore, $\sum_{i=1}^ma_ib_i=\sum_{i=1}^m\sum_jk_{i_j}r_ix^{s_{i_j}}x^{t_i}=0$, a contradiction. Hence $x^n$ is not in $\mfr(R,G)$ for all $n\geq 1$. It follows that $\mathfrak{R}(R,G)\cong \kk[x]$.

(ii) Let $R=\kk_{-1}[x,y]$. Let $\sigma$ be the automorphism of $R$  defined by $\sigma(x)=y$ and $\sigma(y)=x$. Let $G=\{1,\sigma\}$. As we see in Example \ref{example1}, the sequences $(x,y)$ and $(x,y)$ are pertinent, $x^2+y^2\in \mfr(R,G)$. Similar to (i), $x-y\in \mfr(R,G)$. If $char\kk\neq 2$, then we see $x^2,y^2,xy\in \mfr(R,G)$. Therefore $\mfr(R,G)=\kk(x-y)\oplus R_2\oplus R_3\oplus\cdots$, and $\mathfrak{R}(R,G)$ is of dimension 2.
\end{example}

In some cases, the $G$-radical may be trivial. For example, if the $G$-action on $R$ is trivial, that is, $g\cdot a=a$ for all $g\in G$ and $a\in R$, then one sees that the $G$-radical is zero. Following the classical radical theory, we introduce the following definition.

\begin{definition} Let $R$ be a left $G$-module algebra. If the $G$-radial $\mfr(R,G)=0$, then we say that the $G$-action $R$ is {\it $G$-semisimple}.
\end{definition}

By the definition of the $G$-radical, one sees the following basic property of $G$-semisimple $G$-module algebras.

\begin{proposition} \label{prop-subalg} Let $R$ be a left $G$-semisimple $G$-module algebra. If $S$ is a subalgebra of $R$ closed under the $G$-action, then $S$, as a $G$-module algebra, is $G$-semisimple.
\end{proposition}

\section{Pertinency algebras} \label{sec2}

Let $G$ be a group, and let $R$ be a $G$-module algebra. We will review some basic properties of the $G$-radical of $R$, which justify the name ``radical''.

\begin{proposition} Let $R$ be a $G$-module algebra with $G$-action $\theta:G\to \Aut(R)$.
\begin{itemize}
  \item [(i)] The $G$-radical $\mfr(R,G)$ is a $G$-ideal, i.e., it is closed under the $G$-action.
  \item[(ii)] The $G$-action $\theta$ induces a $G$-action $\overline{\theta}:G\to \Aut(\mathfrak{R}(R,G))$ so that $\mathfrak{R}(R,G)$ is also a left $G$-module algebra.
  \item [(iii)] If $G$ is a finite group, then $\mfr(\mathfrak{R}(R,G),G)=0$, i.e., the $G$-module algebra $\mathfrak{R}(R,G)$ is $G$-semisimple.
\end{itemize}
\end{proposition}
\begin{proof} The statement (i) follows from Proposition \ref{prop-basic}(i). Then statement (ii) follows from (i).

(iii) For an element $a\in R$, we write $\overline{a}$ for the image of $a$ in $\mathfrak{R}(R,G)$. Let $(a_1,\dots,a_n)$ and $(b_1,\dots,b_n)$ be sequences of elements in $R$ such that $(\overline{a}_1,\dots,\overline{a}_n)$ and $(\overline{b}_1,\dots,\overline{b}_n)$ are pertinent under the $G$-action in $\mathfrak{R}(R,G)$. Since $G$ is a finite group, we assume $|G|=k$ and we may label the elements of $G$ in $\{1=g_0,g_1,\dots,g_{k-1}\}$.
Then we have $\sum_{i=1}^na_i(g_j\cdot b_i)\in \mfr(R,G)$ for all $j=1,\dots,k-1$. Hence, for each $1\leq j\leq k-1$, there are pertinent sequences $(x_{j1},\dots,x_{jm_j})$ and $(y_{j1},\dots,y_{jm_j})$ in $R$ such that $\sum_{i=1}^na_i(g_j\cdot b_i)=\sum_{i=1}^{m_j}x_{ji}y_{ji}$. We claim that the sequences $(a_1,\dots,a_n)\vee(-x_{11},\dots, -x_{1m_1})\vee\cdots\vee(-x_{k-1,1},\dots, -x_{k-1,m_{k-1}})$ and $(b_1,\dots,b_n)\vee(g_1^{-1}\cdot y_{11},\dots,g_1^{-1}\cdot y_{1m_1})\vee\cdots\vee(g_{k-1}^{-1}\cdot y_{k-1,1},\dots,g_{k-1}^{-1}\cdot y_{k-1,m_{k-1}})$ are pertinent. Indeed, for $1\leq j\leq k-1$, we have
$$\begin{array}{cl}
  &\displaystyle\sum_{i=1}^na_i(g_j\cdot b_i)-\sum_{t=1}^{k-1}\sum_{s=1}^{m_t}x_{ts}(g_jg_t^{-1}\cdot y_{ts})\\
  =&\displaystyle\sum_{i=1}^na_i(g_j\cdot b_i)-\sum_{s=1}^{m_j}x_{js}y_{js}-\sum_{t\neq j}\sum_{s=1}^{m_t}x_{ts}(g_jg_t^{-1}\cdot y_{ts})\\
  =&0.
\end{array}$$
Hence we have $\displaystyle\sum_{i=1}^na_ib_i-\sum_{t=1}^{k-1}\sum_{s=1}^{m_t}x_{ts}(g_t^{-1}\cdot y_{ts})\in \mfr(R,G)$. Since $(x_{t1},\dots,x_{tm_t})$ and $(y_{t1},\dots,y_{tm_t})$ are pertinent for all $t=1,\dots,k-1$, we have $\displaystyle\sum_{t=1}^{k-1}\sum_{s=1}^{m_t}x_{ts}(g_t^{-1}\cdot y_{ts})=0$. Hence $\sum_{i=1}^na_ib_i\in \mfr(R,G)$, that is, $\sum_{i=1}^n\overline{a}_i\overline{b}_i=0$ in $\mathfrak{R}(R,G)$. It follows that the $G$-module algebra $\mathfrak{R}(R,G)$ is $G$-semisimple.
\end{proof}

By the third statement of the proposition above, we see that the pertinency algebra $\mathfrak{R}(R,G)$ is always $G$-semisimple for any finite group action. Indeed, the $G$-radical $\mfr(R,G)$ is the smallest ideal of $R$ closed under $G$-action such that the the quotient algebra is $G$-semisimple.

\begin{lemma} \label{lem-min} Let $G$ be a finite group acting on $A$ by automorphisms. If $I$ is a $G$-ideal of $R$ such that $R/I$ is $G$-semisimple, then $\mfr(R,G)\subseteq I$.
\end{lemma}
\begin{proof} For an element $a\in R$, we write $\overline{a}$ for the image of $a$ in $R/I$. Let $(a_1,\dots,a_n)$ and $(b_1,\dots,b_n)$ be pertinent sequences in $R$. Then $\sum_{i=1}^na_i(g\cdot b_i)=0$ for all $1\neq g\in G$. Hence $\sum_{i=1}^n\overline{a_i}(\overline{g\cdot b_i})=\sum_{i=1}^n\overline{a_i}(g\cdot \overline{b_i})=0$ in $R/I$. Since $R/I$ is $G$-semisimple, we have $\sum_{i=1}^n\overline{a_i}\overline{b_i}=0$ in $R/I$. Therefore, $\sum_{i=1}^na_i b_i\in I$, and hence $\mfr(R,G)\subseteq I$.
\end{proof}

We have the following ``universal property'' of the algebra morphisms which respect the group actions.

\begin{proposition} \label{prop-univ} Let $G$ be a finite group, and let $R$ be a left $G$-module algebra. Let $S$ be a left $G$-semisimple $G$-module algebra. For any $G$-module algebra morphism $f:R\to S$, there is a unique $G$-module algebra morphism $g:\mathfrak{R}(R,G)\to S$ such that the following diagram commutes
$$\xymatrix{
  R \ar[d]_{\pi} \ar[r]^{f} &     S,   \\
  \mathfrak{R}(R,G) \ar[ur]_{g}                     }$$ where $\pi$ is the natural projection.
\end{proposition}
\begin{proof} Since $S$ is $G$-semisimple and $f$ is also a $G$-module morphism, the subalgebra im$(f)$ of $S$ is also $G$-semisimple by Proposition \ref{prop-subalg}. Hence the $G$-radical $\mfr(R,G)$ is contained in $\ker(f)$ by Lemma \ref{lem-min}. Thus the assertion follows.
\end{proof}

Let $G$ be a finite group, $R$ a left $G$-module algebra, and $B=R*G$ the skew polynomial algebra.
Let $\iota: R\to B, a\mapsto a\otimes 1$ be the embedding map, so that $R$ is viewed as a subalgebra of $B$. Similarly, we view $\kk G$ as a subalgebra of $B$. Set $e=\frac{1}{|G|}\sum_{g\in G}g\in \kk G$. Then $e$ is an idempotent in $B$. Let $\mathcal{M}_R$ be the class of right $R$-modules obtained from right $B$-modules which are annihilated by $e$. The next proposition shows that the $G$-radical of $R$ is indeed the $\mathcal{M}_R$-radical in the classical radical theory (cf. \cite{Sz}).

\begin{proposition} \label{prop-mrad} With the notions as above, the following holds: $$\mfr(R,G)=\bigcap_{M\in\mathcal{M}_R}Ann_R(M)$$ where $Ann_R(M)$ is the annihilator of $M$ in $R$.
\end{proposition}
\begin{proof} We always view $R$ as a subalgebra of $B$ through the map $\iota$. We see that two sequences $(a_1,\dots,a_n)$ and $(b_1,\dots,b_n)$ are pertinent under the $G$-action if and only if $\sum_{i=1}^na_ieb_i=\sum_{i=1}^na_ib_i\otimes 1$. Note that $BeB=ReR$ (cf. \cite[Proposition 2.13]{CFM}). We have $R\cap BeB=\mfr(R,G)$. For a module $M\in \mathcal{M}_R$, we have $Me=0$, and hence $BeB\subseteq Ann_R(M)$. Therefore $\mfr(R,G)\subseteq Ann_R(M)$. Since $\overline{B}\in \mathcal{M}_R$, we have $Ann_R(\overline{B})=R\cap BeB=\mfr(R,G)$. Therefore the proposition follows.
\end{proof}

The following lemma is easy to check.

\begin{lemma} \label{lem-pertinred} Let $G$ be a finite group, and $R$ a left $G$-module algebra.
\begin{itemize}
  \item [(i)] If $(a_1,\dots,a_n)\overset{G}\sim(c_1b_1,\dots,c_nb_n)$ where $c_1,\dots,c_n\in R^G$. Then $$(a_1c_1,\dots,a_nc_n)\overset{G}\sim(b_1,\dots,b_n).$$
  \item [(ii)] Assume $(a_1,\dots,a_n)\overset{G}\sim(b_1,\dots,b_n)$. If $b_i=b_j$ for some $i,j$, then $$(a_1,\dots,a_i+a_j,\dots,\hat{a}_j,\dots,a_n)\overset{G}\sim(b_1,\dots,b_i,\dots,\hat{b}_j,\dots,b_n),$$ where $\hat{a}_j$ (and $\hat{b}_j$) means deleting the element $a_j$ ($b_j$,resp.).
  \item[(iii)]  Assume $(a_1,\dots,a_n)\overset{G}\sim(b_1,\dots,b_n)$. If $a_i=a_j$ for some $i,j$, then $$(a_1,\dots,a_i,\dots,\hat{a}_j,\dots,a_n)\overset{G}\sim(b_1,\dots,b_i+b_j,\dots,\hat{b}_j,\dots,b_n),$$ where $\hat{a}_j$ (and $\hat{b}_j$) means deleting the element $a_j$ ($b_j$,resp.).
\end{itemize}
\end{lemma}

We end this section with some sufficient conditions for a $G$-module algebra to be $G$-semisimple.

\begin{proposition} \label{prop-ss} Let $G$ be a finite group, and let $R$ be a left $G$-module algebra. Set $A=R^G:=\{a\in R|g\cdot a=a,\forall\  g\in G\}$. If $R$ is a free $A$-module with a free basis $y_0=1,y_1,\dots,y_n\in R$ satisfying the conditions:

{\rm(i)} $g(y_1,\dots,y_n)=(\xi_1y_1,\dots,\xi_ny_n)$ where $\xi_i$ is a $j_i$-th ($j_i\ge3$) primitive root of unity for all $i=1,\dots,n$,

{\rm (ii)} $y_iy_j=0$ for $i,j=1,\dots,n$,\\
 then $R$ is $G$-semisimple.
\end{proposition}
\begin{proof} By Lemma \ref{lem-pertinred}, we may assume that pertinent sequences have the form $(a_0,a_1,\dots,a_n)\overset{G}\sim(1,y_1,\dots,y_n)$. By the condition (ii), we may assume $a_1,\dots,a_n\in A$. Assume $a_0=b_0+b_1y_1+\cdots+b_ny_n$ with $b_i\in A$ for $i=0,\dots,n$. By the equality $\sum_{i=0}^na_ig(y_i)=0$ we obtain $b_0=0$ and $b_i=-\xi_ia_i$ for $i=1,\dots,n$. By the equality $\sum_{i=0}^na_ig^2(y_i)=0$, we obtain $b_i=-\xi_i^2a_i$ for $i=1,\dots,n$. Since $\xi_i$ is a $j_i$-th primitive root of unity with $j_i\ge3$ for $i=1,\dots,n$,  we have  $b_i=a_i=0$ for $i=1,\dots,n$. Hence $\mfr(R,G)=0$.
\end{proof}

\section{Group actions with $\Pty(R,G)=1$}\label{sec3}

Recall that we have a definition of the pertinency of a group action on an algebra in \cite{BHZ}, where we had to assume that the group is finite. In order to avoid this restriction, we provide an alternative (equivalent) definition in this paper. We will see, in some cases, the definition here is more flexible.

\begin{definition} Let $G$ be group, and let $R$ be a noetherian algebra with finite Gelfand-Killilov (GK) dimension. The {\it pertinency} of the $G$-action on $R$ is defined to be the number
\begin{equation}\label{pty1}
  \Pty(R,G)=\GKdim (R)-\GKdim (\mathfrak{R}(R,G)).
\end{equation}

\end{definition}

Let $R$ be a noetherian $G$-module algebra with finite GK-dimension, and let $B:=R*G$ be the skew group algebra. Assume $G$ is finite. Let $\kk G$ be the group algebra. As before, let $e=\frac{1}{|G|}\sum_{g\in G}g$. Recall from \cite{BHZ} that the pertinency of the $G$-action is defined to be
\begin{equation}\label{pty2}
  \Pty(R,G)=\GKdim (R)-\GKdim(B/(BeB)).
\end{equation}

\begin{lemma}\label{lem1} Assume $G$ is a finite group. Let $R$ be a noetherian left $G$-module algebra with finite GK-dimension. Then $\GKdim(B/(BeB))=\GKdim (\mathfrak{R}(R,G))$.
\end{lemma}
\begin{proof} As before, let $\iota:R\to B,a\mapsto a\otimes 1$ be the inclusion map. Let $\pi:B\to B/BeB$ be the projection map.
Then $\overline{R}:=\pi\circ\iota(R)$ is a subalgebra of $B/BeB$. Since $B$ is finitely generated as a left $R$-module, we see that $B/BeB$ is a finitely generated $\overline{R}$-module.  By \cite[Proposition 5.5]{KL}$, \GKdim(\overline{R})=\GKdim(B/BeB)$. As in the proof of Proposition \ref{prop-mrad}, we see $\mfr(R,G)=R\cap BeB$. Therefore $\overline{R}\cong \mathfrak{R}(R,G)$.
\end{proof}
 As a consequence of Lemma \ref{lem1}, we obtain:

\begin{proposition} With the assumptions in Lemma \ref{lem1}, the pertinencies defined by (\ref{pty1}) and (\ref{pty2}) are equivalent.
\end{proposition}

In order to compute the pertinency of a group action, one has to analyze the radical ideal of the group action. Next we provide several ways to find elements in the radical ideal.

\begin{lemma} \label{lem2} Let $R$ be an algebra, and $G$ a cyclic group generated by $\sigma$. Assume that there are elements $a_1,\dots,a_n\in R$ such that $\sigma\cdot a_i=\xi a_i$ for $i=1,\dots,n$, where $\xi$ is an $n$-th primitive root of unity. If $\ch\kk\nmid n$, then $a_1a_2\cdots a_n\in \mfr(R,G)$.
\end{lemma}
\begin{proof} We show that the following sequences are pertinent under the $G$-action: $(1,a_1,a_1a_2,\dots,a_1a_2\cdots a_{n-1})$ and $(a_1a_2\cdots a_n,a_2\cdots a_n,\dots,a_n)$. Note that $\sigma^j\cdot a_i=\xi^j a_i$. For $j=1,\dots,n-1$, we have $\sum_{i=0}^{n-1}a_1\cdots a_i(\sigma^j\cdot (a_{i+1}\cdots a_n))=(\sum_{i=0}^{n-1}\xi^{j(n-i)})a_1a_2\cdots a_n=0$. Hence $a_1a_2\cdots a_n\in \mfr(R,G)$.
\end{proof}

Let $R$ be an algebra, and let $a_1,\dots,a_k\in R$. We write $a_1\cdots \hat{a}_i\cdots a_k$ for the product $a_1\cdots a_{i-1}a_{i+1}\cdots a_k$ with $a_i$ missing. Similarly, we have the notion $a_1\cdots \hat{a}_{i_1}\cdots \hat{a}_{i_2}\cdots \hat{a}_{i_s}\cdots a_k$ for $1\leq i_1<\dots<i_s\leq k$. We write $$a_{\widehat{[i_1i_2\cdots i_s]}}:=a_1\cdots \hat{a}_{i_1}\cdots \hat{a}_{i_2}\cdots \hat{a}_{i_s}\cdots a_k,$$
and
$$(a_{\widehat{[i_1i_2\cdots i_s]}}:1\leq i_1<\dots<i_s\leq k)$$
for the sequence corresponding to all the possible choices of arrays $[i_1i_2\cdots i_s]$ in the lexicographic order. Similarly, we have $$a_{[i_1i_2\cdots i_s]}:=a_{i_1}\cdots a_{i_s}$$ for the products of elements, and
$$(a_{[i_1i_2\cdots i_s]}:1\leq i_1<\dots<i_s\leq k).$$
Let $G$ be a group acting on $R$. For $g_1,\dots,g_k\in G$, we write
{\small$$g_{\widehat{[i_1i_2\cdots i_s]}}(a_{\widehat{[i_1i_2\cdots i_s]}}):=g_1(a_1)\cdots \hat{a}_{i_1}\cdots \hat{a}_{i_2}\cdots \hat{a}_{i_s}\cdots g_k(a_k),$$}
and $$\left(g_{\widehat{[i_1i_2\cdots i_s]}}(a_{\widehat{[i_1i_2\cdots i_s]}}):1\leq i_1<\dots<i_s\leq k\right)$$ for the corresponding sequences in the lexicographic order.

\begin{lemma} \label{lem2-1} Let $R$ be an algebra, and  $G=\{g_1,\dots,g_{n-1},g_n=1\}$  a finite group acting on $R$. Assume that $a_1,\dots,a_{n-1}\in R$ are central elements of $R$. Then the following two sequences $$\bigcup_{s=0}^{n-1}\left(g_{\widehat{[i_1i_2\cdots i_s]}}(a_{\widehat{[i_1i_2\cdots i_s]}}):1\leq i_1<\cdots< i_s\leq n-1\right)$$ and $$\bigcup_{s=0}^{n-1}((-1)^sa_{[i_1i_2\cdots i_s]}:1\leq i_1<\cdots< i_s\leq n-1)$$ are pertinent, where $a_{[\ ]}=1$ and $a_{\widehat{[\ ]}}=a_1a_2\cdots a_{n-1}$ when $s=0$.
\end{lemma}
\begin{proof} For every $1\neq g\in G$, we have $$\sum_{[i_1i_2\cdots i_s]} g_{\widehat{[i_1i_2\cdots i_s]}}(a_{\widehat{[i_1i_2\cdots i_s]}})g(a_{[i_1i_2\cdots i_s]})= \prod_{i=1}^{n-1}(g_i(a_i)-g(a_i)),$$ where on the left hand side, the sum is over all the possible choices of the arrays $[i_1i_2\cdots i_s]$ for $s=0,\dots,n-1$. Since $g$ must be one of the elements in $\{g_1,\dots, g_{n-1}\}$, the right hand side of the above equation is zero. Hence the result follows from the definition of the pertinent sequences.
\end{proof}

Now let $R$ be an algebra, and let $a_1,\dots,a_{n-1}\in R$ such that $a_ia_j=q_{ij}a_ja_i$ for all $i< j$ and some $q_{ij}\neq 0$. Let $G=\{g_1,\dots,g_{n-1},g_n=1\}$. Assume $g_i(a_i)=\xi_i a_i$ for $i=1,\dots,n-1$ with $\xi_i\in\kk$. Write
$$q_{[i_1i_2\cdots i_s]}=\prod_{k=1}^s\prod_{\varepsilon} q_{i_kj},$$ where the restriction $\varepsilon$ under the second $\prod$ reads  as $\{j>i_k,\hbox{ and }j\neq i_1,\dots, i_s\}$.

\begin{lemma} \label{lem2-2} With the assumptions as above, we have the following pertinent sequences $$\bigcup_{s=0}^{n-1}g_{\widehat{[i_1i_2\cdots i_s]}}(a_{\widehat{[i_1i_2\cdots i_s]}}:1\leq i_1<\cdots<i_s\leq n-1)$$ and $$\bigcup_{s=0}^{n-1}((-1)^sq_{[i_1i_2\cdots i_s]}a_{[i_1i_2\cdots i_s]}:1\leq i_1<\cdots<i_s\leq n-1),$$ where $a_{[\ ]}=1$ and $a_{\overline{[\ ]}}=a_1a_2\cdots a_{n-1}$ when $s=0$.
\end{lemma}
\begin{proof} For every $1\neq g\in G$, since $g$ acts on $a_1,\dots,a_{n-1}$ diagonally, we have $a_{\widehat{[i_1i_2\cdots i_s]}}g((-1)^sq_{[i_1i_2\cdots i_s]}a_{[i_1i_2\cdots i_s]})=(-1)^s a_1\cdots g(a_{i_1})\cdots g(a_{i_s})\cdots a_{n-1}$. Then the same arguments as in the proof of Lemma \ref{lem2-1} hold.
\end{proof}

As a consequence, we have the following result, a part of which is indicated in \cite[Lemma 2.2]{BL} with different constructions.

\begin{proposition} \label{prop2-1} With the same hypothesis as in Lemma \ref{lem2-1}, we have $$\prod_{i=1}^{n-1}(g_i(a_i)-a_i)\in \mfr(R,G).$$
\end{proposition}
\begin{proof} Note that the sum of the products of the corresponding elements in the pertinent sequences $\bigcup_{s=0}^{n-1}g_{\overline{[i_1i_2\cdots i_s]}}(a_{\overline{[i_1i_2\cdots i_s]}})$ and $\bigcup_{s=0}^{n-1}(a_{[i_1i_2\cdots i_s]})$ is exactly equal to $\prod_{i=1}^{n-1}(g_i(a_i)-a_i)$.
\end{proof}

There is another way to construct elements in the radical of a group action by means of the determinant.

Let $G=\{1=g_0,g_1,\dots,g_{n-1}\}$ be a finite group acting on an algebra $R$. Taking any elements $a_1,\dots,a_n\in Z(R)$, we define an element
\begin{equation}\label{eq-det}
  \delta_G(a_1,\dots,a_n)=\begin{vmatrix}
a_1&a_2&\cdots& a_n\\
g_1(a_1)&g_1(a_2)&\cdots&g_1(a_n)\\
\vdots&\vdots&&\vdots\\
g_{n-1}(a_1)&g_{n-1}(a_2)&\cdots&g_{n-1}(a_n)
\end{vmatrix}.
\end{equation}

\begin{proposition}\label{prop2-6} With the notions as above, we have $\delta_G(a_1,\dots,a_n)\in \mfr(R,G)$.
\end{proposition}
\begin{proof} Consider the following two sequences $(a_1,\dots,a_n)$ and $(A_1,\dots,A_n)$, where $A_i$ $(i=1,\dots n)$ is the cofactor of the element $a_i$ in the determinant (\ref{eq-det}). For every $1\neq g\in G$, we have $gg_j=1$ for some $1\leq j\leq n-1$. Hence $$g(A_i)=\begin{vmatrix}
gg_1(a_1)&\cdots&gg_1(a_{i-1})&gg_1(a_{i+1})&\cdots&gg_1(a_n)\\
\vdots&&\vdots&\vdots&&\vdots\\
a_1&\cdots&a_{i-1}&a_{i+1}&\cdots& a_n\\
\vdots&&\vdots&\vdots&&\vdots\\
gg_{n-1}(a_1)&\cdots&gg_{n-1}(a_{i-1})&gg_{n-1}(a_{i+1})&\cdots&gg_{n-1}(a_n)
\end{vmatrix}.$$
Then $\sum_{i=1}^n(a_ig(A_i))$ is a determinant in which there are tow equal rows, and hence it is zero. Therefore, $\delta_G(a_1,\dots,a_n)=\sum_{i=1}^na_iA_i\in \mfr(R,G)$.
\end{proof}

\begin{proposition} \label{prop2-2} Let $R$ be a noetherian algebra which is a domain with $1\leq \GKdim(R)<\infty$, and let $G\subseteq \Aut(R)$ be a finite subgroup with $|G|\neq 1$ and $\ch\kk\nmid |G|$. If $R$ and $G$ satisfy one of the following conditions, then $\Pty(R,G)\ge1$.
\begin{itemize}
  \item [(i)] $R$ is connected graded, the $G$-action preserves the grading of $R$, and the order of $G$ is prime.
  \item [(ii)] $G$ acts faithfully on the center of $R$.
  \item [(iii)] $R=\kk_{q_{ij}}[x_1,\dots,x_n]$, $|G|\leq n+1$ and every $g\in G$ acts on $x_1,\dots,x_{n}$ diagonally.
\end{itemize}
\end{proposition}
\begin{proof} (i) If $|G|$ is a prime number, then $G$ is a cyclic group. Assume $G=\langle g\rangle$. Since $R$ is noetherian, $\dim R_i<\infty$ for all $i\ge0$. Since the order of $g$ is finite, we may choose a basis $S$ of $R_1$ so that $g$ acts on the basis $S$ diagonally. Thus for any element $a\in S$, $g(a)=\xi a$ for some $p$-th root of unit. Letting all the elements $a_i=a$ in Lemma \ref{lem2}, we see $a^p\in \mfr(R,G)$. Since $R$ is a domain, $a^p\neq 0$. It follows that $\GKdim(R/\mfr(R,G))\leq \GKdim(R/Ra^pR)\leq \GKdim(R)-1$. Hence $\Pty(R,G)\ge1$.

(ii) Let $Z(R)$ be the center of $R$. Since $G$ acts on $Z(R)$ faithfully, for any $1\neq g\in G$, there is element $a\in Z(R)$ such that $g(a)\neq a$. Since $R$ is a domain, Proposition \ref{prop2-1} asserts that there is a nonzero element $b\in\mfr(R,G)$. The rest of the proof is similar to the case (i).

(iii) Similar to the above case, the assertion follows from Proposition \ref{prop2-1}.
\end{proof}

The above result shows that in reasonable cases the pertinency of a finite group action is nonzero, which implies the following primeness of the skew group algebra.

\begin{proposition} \label{prop2-3} Let $R$ be a neotherian locally finite graded algebra which is a domain with $1\leq \GKdim(R)<\infty$. Let $G$ be a nontrivial finite subgroup of $\Aut(R)$ such that the restriction of $G$ to the center of $R$ is faithful. Then the skew group algebra $R*G$ is prime.
\end{proposition}
\begin{proof} By \cite[Lemma 3.10]{BHZ}, $R*G$ is prime if and only if $\Pty(G,R)\ge1$. Now the result follows from Proposition \ref{prop2-2}.
\end{proof}

The finite group actions of pertinency 1 is of special interest. If $R$ is a polynomial ring, then we have the following version of Shephard-Todd-Chevalley Theorem.

\begin{proposition} \label{prop2-4} Let $V$ be a finite dimensional vector space, and $R=\kk[V]$ the polynomial algebra. Let $G\subseteq GL(V)$ be a finite subgroup. Then for the invariant subalgebra $R^G$, $\gldim (R^G)<\infty$ if and only if $\Pty(R,G)=1$.
\end{proposition}
\begin{proof} By \cite[Lemma 7.2]{BHZ2}, we see that $\gldim(R^G)<\infty$ if and only if $\Pty(R,G)\leq1$. Proposition \ref{prop2-2} applies to obtain the desired result.
\end{proof}

If $R$ is not commutative, we have the following result.

\begin{proposition}\label{prop2-5} Let $R$ be a neotherian AS-regular algebra which is a domain and Cohen-Macaulay, and let $G\subseteq\Aut(R)$ be a finite subgroup. Assume that $R$ and $G$ are in one of the cases in Proposition \ref{prop2-2}. If $\gldim(R^G)<\infty$, then $\Pty(R,G)=1$.
\end{proposition}
\begin{proof} By \cite[Lemmas 1.10 and 1.11]{KKZ}, if $R^G$ has finite global dimension, then $R$ is free as a graded right $R^G$-module. Since both $R$ and $R^G$ are connected graded, $R$ can be written as $R\cong R^G\oplus M$ where $M\cong\oplus_{s=1}^n R^G(i_s)$ for some $n$ and $i_s\leq -1$, where $R^G(i_s)$ is a graded $R^G$-module by a shift of degree $i_s$. Hence as a graded algebra $\End_{R^G}(R)$ contains nonzero component of negative degree. So, $\End_{R^G}(R)$ can not be isomorphic to $R*G$ as graded algebras. By \cite[Theorem 0.3]{BHZ}, $\Pty(R,G)\leq 1$. It follows from Proposition \ref{prop2-2} that $\Pty(R,G)=1$.
\end{proof}

\section{Group actions with $\Pty(R,G)>1$} \label{set4}

Let $G$ be a finite group, and let $R$ be a noetherian left $G$-module algebra. It is possible that the pertinency algebra $\mathfrak{R}(R,G)$ is trivial, that is, $\mfr(R,G)=R$. Indeed, this is the case when $R/R^G$ is a {\it Hopf Galois extension} (for the definition, see \cite{CFM}).

\begin{proposition} \label{prop3-1} Let $G$ be a finite group and let $R$ be a left $G$-module algebra. Then $\mfr(R,G)=R$ if and only if $R/R^G$ is a Hopf Galois extension.
\end{proposition}
\begin{proof} Let $e=\frac{1}{|G|}\sum_{g\in G} g\in \kk G$. Note that $\mfr(R,G)=R$ if and only if $ReR=R*G$, if and only if $R^G$ and $R$ are Morita equivalent, which is equivalent to the condition that $R/R^G$ is Hopf Galois extension \cite[Theorem 1.2]{CFM}.
\end{proof}

Assume that $R$ has finite GK-dimension. It is clear that if $R/R^G$ is a Hopf Galois extension, then $\Pty(R,G)=\GKdim(R)$. A weak version of a Hopf Galois extension, called a {\it Hopf dense Galois extension}, was introduced in \cite{HVZ}. The next result is a consequence of Lemma \ref{lem1} and \cite[Proposition 1.3]{HVZ}.

\begin{proposition} Assume that $G$ is a finite group. Let $R$ be a noetherian left $G$-module algebra with finite GK-dimension. Then $R/R^G$ is a Hopf dense Galois extension if and only if $\Pty(R,G)=\GKdim(R)$.
\end{proposition}

The following result gives an example of Hopf dense Galois extensions.

\begin{proposition} Let $R=\kk\oplus R_1\oplus R_2\oplus\cdots$ be a neotherian connected graded algebra generated by elements $x_1,\dots,x_m$ of degree one with finite GK-dimension. Let $G$ be the cyclic group generated by the automorphism defined by $\sigma(x_i)=\xi x_i$ for $i=1,\dots,m$, where $\xi$ is an $n$-th primitive root of unity.  If $\ch\kk\nmid n$, then $\Pty(R,G)=\GKdim(R)$.
\end{proposition}
\begin{proof} For any elements $a_1,\dots,a_n\in R_1$, we see $\sigma(a_i)=\xi a_i$ for $i=1,\dots,n$. By Lemma \ref{lem2}, we have $a_1a_2\cdots a_n\in \mfr(R,G)$. Note that $R_k=(R_1)^k$ for all $k\ge1$. Then we see $R_n\subseteq \mfr(R,G)$, and hence $R_i\subseteq \mfr(R,G)$ for all $i\ge n$. Therefore the pertinency algebra $\mathfrak{R}(R,G)$ is finite dimensional. Hence $\Pty(R,G)=\GKdim (R)$.
\end{proof}

Let $R=\kk_{-1}[x_1,\dots,x_n]$ $(n\ge2)$ be the skew-symmetric algebra. Let $\sigma$ be the automorphism of $R$ defined by $\sigma(x_i)=x_{i+1}$ for $i=1,\dots,n-1$ and $\sigma(x_n)=x_1$. Let $G\leq \Aut(R)$ be the subgroup generated by $\sigma$. The pertinency $\Pty(R,G)$ was computed in \cite{BHZ}. We give an alternative computation in this paper.

We may choose another set of generators for $R$ so that $\sigma$ acts on the generators diagonally. Let $\xi$ be an $n$-th primitive root of unity. For $j=1,\dots,n$, let $y_j=\sum_{i=1}^n\xi^{ji}x_i$. Then $\{y_1,\dots,y_n\}$ is a set of linear independent generators of $R$, and $\sigma(y_j)=\xi^{-j}y_i$ for $j=1,\dots,n$. For $1\leq k\leq n-1$, we have $$\sum_{i=0}^{n-1}y_j^i\sigma^k(y_j^{n-i})=\sum_{i=0}^{n-1}\xi^{-jk(n-i)}y_j^n.$$ If $\gcd(j,n)=1$, then $n\nmid jk$. In this case, $\sum_{i=0}^{n-1}\xi^{-jk(n-i)}=0$, and hence $\sum_{i=0}^{n-1}y_j^i\sigma^k(y_j^{n-i})=0$, which in turn implies that the sequences $(1,y_j,\dots,y_j^{n-1})$ and $(y_j^n,y_j^{n-1},\dots,y_j)$ are pertinent under the $G$-action. Hence we have the following result.

\begin{lemma} \label{lem3} With the notations as above, if $\ch\kk\nmid n$ and $\gcd(j,n)=1$, then $y_j^n\in \mfr(R,G)$.
\end{lemma}

Let $A$ be the subalgebra of $R$ generated by $x_1^2,\dots,x^2_n$. Then $A$ is a subalgebra of the center of $R$. Note that $R$ is finitely generated as an $A$-module. For $j=1,\dots, n$, let $Y_j=\sum_{i=1}^n\xi^{ij}x_i^2$. Then $Y_1,\dots,Y_n$ are linearly independent generators of $R$ (cf. \cite{BHZ}). Let $A'=A/(A\bigcap \mfr(R,G))$.  It is clear that $A'$ is a subalgebra of $\mathfrak{R}(R,G)$ and $\mathfrak{R}(R,G)$ is finitely generated as an $R'$-module since $R$ is finitely generated as an $A$-module. So, $\GKdim (\mathfrak{R}(R,G))=\GKdim (R')$. For $j=1,\dots,n$, we have $y_j^2=\sum_{i=1}^n\xi^{2ij}x_i^2$. For $n+1\leq t\leq 2n$, we set $Y_t=Y_{t-n}$. Then we see $y^2_j=Y_{2j}$ for $j=1,\dots,n$. If $\gcd(j,n)=1$, then $y_j^n\in \mfr(R,G)$ by Lemma \ref{lem3}, which implies $y_j^{2n}\in \mfr(R,G)$, and hence $Y_{2j}^n\in \mfr(R,G)$. So, we obtain that $Y_{2j}^n=0$ in $R'$ if $\gcd(j,n)=1$. Now let $T$ be the subalgebra of $A'$ generated by elements of the set $\{Y_1,\dots,Y_n\}\setminus\{Y_{2j}|\gcd(j,n)=1\}$. Then $A'$ is finitely generated as a $T$-module. Hence $\GKdim(A')=\GKdim(T)$. Let $t$ be the cardinality of the set $\{Y_{2j}|\gcd(j,n)=1\}$. The commutative subalgebra $T$ is generated by $n-t$ elements. Hence $\GKdim(T)\leq n-t$. Hence $$\Pty(R,G)=n-\GKdim(\mathfrak{R}(R,G))=n-\GKdim(T)\ge t.$$ Note that $$t=\begin{cases} \phi(n), & \hbox{if }4\nmid n;\\
\frac{\phi(n)}{2}, &\hbox{if } 4|n,\end{cases} $$  where $\phi(n):=n \prod_{{\text{all primes $p\mid n$}}}
(1-\frac{1}{p})$ is the Euler's totient function (cf. \cite{BHZ}).

Summarizing the above narratives, we have provided a simpler proof for \cite[Theorem 5.7(iii)]{BHZ}.

\begin{theorem}\cite{BHZ} Assume $\ch\kk\nmid n$ $(n\ge2)$. Let $R=\kk_{-1}[x_1,\dots,x_n]$, and $G$ the subgroup of $\Aut(A)$ generated by the automorphism $\sigma$ defined by $\sigma(x_i)=x_{i+1}$ for $1\leq i\leq n-1$ and $\sigma(x_n)=x_1$. Then we have $$\Pty(R,G)\ge \begin{cases} \phi(n), &\hbox{if }4\nmid n;\\
\frac{\phi(n)}{2}, & \hbox{if }4\mid n,\end{cases}$$ where $\phi(n):=n \prod_{{\text{all primes $p\mid n$}}}
(1-\frac{1}{p})$ is the Euler's totient function.
\end{theorem}

Now we consider the down-up algebra: $R=\kk\langle x,y\rangle/(r_1,r_2)$, where $r_1=x^2y-\alpha xyx-\beta yx^2$ and $r_2=x y^2-\alpha yxy-\beta y^2 x$. The pertinencies of finite group actions on $R$ were computed in \cite{BHZ2} when $\beta\neq-1$ or $\beta=-1$ but $\alpha=2$. We have the following result.

\begin{proposition} Let $R$ be a down-up algebra such that $\beta=-1$ and $\alpha\neq 2$. Let $G=\{1,\sigma\}$ where $\sigma$  is the automorphism of $R$ defined by $\sigma(x)=ay$ and $\sigma(y)=a^{-1}x$ with $a\neq 0$. If $\ch \kk=0$, then $\Pty(R,G)=3$.
\end{proposition}
\begin{proof} We see that the sequences $(1,-x)$ and $(ay,1)$ are pertinent under the $G$-action. Hence $ay-x\in \mfr(R,G)$. Therefore $x=ay$ in the pertinency algebra $\mathfrak{R}(R,G)$. By definition, $x^2y=\alpha xyx-yx^2$ in $R$, and hence it holds in $\mathfrak{R}(R,G)$. Combining the relations $x=ay$ and $x^2y=\alpha xyx-yx^2$ in $\mathfrak{R}(R,G)$, we see $y^3=0$ in $\mathfrak{R}(R,G)$ since $\alpha\neq 2$. Therefore, $\mathfrak{R}(R,G)$ is finite dimensional. Hence $\Pty(R,G)=3$ as $\GKdim (R)=3$ (cf. \cite{KMP}).
\end{proof}

\begin{remark} The pertinency of a permutation group acting on $\kk_{-1}[x_1,\dots, x_n]$ or on a down-up algebra was recently computed by Gaddis, Kirkman, Moore and Won using different methods \cite{GKMW}.
\end{remark}

\section{Local cohomology}\label{sec5}
In this section, $A$ is always a noetherian algebra, and $\mfa$ is an ideal of $A$.
Denote by $\Mod A$ the category of right $A$-modules, by $\Mod A^\circ$ the category of left $A$-modules, and by $\Mod A^e$ the category of $A$-$A$-bimodules. Let $\Gamma_\mfa=\underrightarrow{\lim}\Hom_A(A/\mfa^n,-):\Mod A\longrightarrow \Mod A$.
For $M\in\Mod A$, if $\Gamma_\mfa(M)=M$, then we say that $M$ is an {\it $\mfa$-torsion module}. If $\Gamma_\mfa(M)=0$, then we say that $M$ is {\it $\mfa$-torsion-free}.

We say that $\mfa$ has the {\it right Artin-Rees (AR) property} if one of the following equivalent conditions holds \cite[4.2.3]{MR}
\begin{itemize}
  \item [(i)] For every right ideal $\mfb$ of $A$, $\mfb\cap \mfa^n\subseteq \mfb\mfa$ for some $n$.
  \item [(ii)] For every finitely generated right $A$-module $M$, and every submodule $N\subseteq M$, $N\cap M\mfa^n\subseteq N\mfa$ for some $n$.
  \item [(iii)] For every finitely generated right $A$-module $M$, and every submodule $N\subseteq M$, and for every integer $s\ge0$, there is an integer $n>0$ such that $N\cap M\mfa^n\subseteq N\mfa^s$.
\end{itemize}

\begin{lemma}\label{lem-torinj} Let $A$ be a noetherian algebra, and $\mfa$ an ideal of $A$.  Let $I$ be the injective envelope of $N\in\Mod A$.
\begin{itemize}
  \item [(i)] If $N$ is $\mfa$-torsion-free, so is $I$.
  \item [(ii)] Assume further that $\mfa$ has the right AR-property. If $N$ is an $\mfa$-torsion module, so is $I$.
\end{itemize}
\end{lemma}
\begin{proof} Since $I$ is the injective envelope of $N$, $N$ is an essential submodule of $I$. If $\Gamma_\mfa(I)\neq0$, then $\Gamma_\mfa(I)\cap N\neq 0$. So, if $N$ is $\mfa$-torsion-free, $\Gamma_\mfa(I)$ has to be zero. The statement (i) follows.

(ii) Take an element $x\in I$. Let $K=N\cap xA$. Then $K$ is an essential submodule of $xA$. Since $A$ is noetherian, $K$ is finitely generated. Since $N$ is an $\mfa$-torsion module, so is $K$. Thus $K\mfa^n=0$ for some $n>0$. By the AR-property, there is an integer $s$ such that $K\cap x\mfa^s\subseteq K\mfa^n=0$. As $K$ is essential in $xA$, it follows that $x\mfa^s=0$. Therefore, $\Gamma_\mfa(I)=I$.
\end{proof}

The functor $\Gamma_\mfa$ is left exact. We write the $i$-th right derived functor as $$R^i\Gamma_\mfa=\underrightarrow{\lim}\Ext_A^i(A/\mfa^n,-).$$ Let $M$ be a right $A$-module. Define $$\depth_\mfa(M)=\inf\{i|R^i\Gamma_\mfa(M)\neq0\}\subseteq\mathbb{N}\cup\{\infty\}.$$

\begin{lemma} \label{lem-injres} Let $M$ be a right $A$-module. Assume $\depth_\mfa(M)=d>0$. Let $0\to M\to I^0\overset{\delta^0}\to I^1\overset{\delta^1}\to\cdots \overset{\delta^n-1}\to I^n\overset{\delta^n}\to\cdots$ be a minimal injective resolution of $M$. Then $I^i$ is $\mfa$-torsion-free for $i<d$.
\end{lemma}
\begin{proof} Since $d>0$, $M$ is $\mfa$-torsion-free. Hence $I^0$ is $\mfa$-torsion-free. Assume $I^i$ $(i<d-1)$ is $\mfa$-torsion-free. If $\Gamma_\mfa(I^{i+1})\neq0$, then $\Gamma_\mfa(I^{i+1})\cap \ker\delta^{i+1}\neq0$. Since $I^i$ is torsion free, $R^{i+1}\Gamma_\mfa(M)=\Gamma_\mfa(I^{i+1})\cap\ker\delta^{i+1}\neq0$, which contradicts with the hypothesis that $\depth_\mfa(M)=d$.
\end{proof}

\begin{lemma} \label{lem-sum} The derived functor $R^i\Gamma_\mfa$ $(i\ge0)$ commutes with direct sums.
\end{lemma}
\begin{proof} Note that $R^i\Gamma_\mfa=\underrightarrow{\lim}\Ext^i_A(A/\mfa^n,-)$. Since $A$ is noetherian, $\Ext^i_A(A/\mfa^n,-)$ commutes with direct sums. The direct limit $\underrightarrow{\lim}$ also commutes with direct sums. Hence $R^i\Gamma_\mfa$ commutes with direct sums.
\end{proof}

Let $B$ be another algebra, and let $M$ be a $B$-$A$-bimodule. Then $\Gamma_\mfa(M)$ is a $B$-$A$-bimodule. By taking the injective resolution of the $B$-$A$-bimodule $M$, we see that $R^i\Gamma_\mfa(M)$ is a $B$-$A$-bimodule for every $i\ge0$. In particular, $R^i\Gamma_\mfa(A)$ is an $A$-$A$-bimodule for every $i\ge0$.

\begin{corollary} \label{cor-tens} Let $B$ be an algebra, and let $P$ be a projective right $B$-module. For every $B$-$A$-bimodule $M$. We have $R^i\Gamma_\mfa(P\otimes_B M)=P\otimes_BR^i\Gamma_\mfa(M)$ for every $i\ge0$.
\end{corollary}

We write $D(\Mod A)$ for the (unbounded) derived category of right $A$-modules, and $D^b(\Mod A)$ (resp. $D^-(\Mod A)$) for the bounded (resp. bounded above) derived category of right $A$-modules. $M^\cdot$ stands for an object in the derived category of right $A$-modules.

\begin{theorem}\label{thm-locdual} Let $A$ be a noetherian algebra, and let $\mfa$ be an ideal of $A$. Assume that $\Gamma_\mfa$ has finite cohomological dimension, or $\id{}_AA<\infty$ and $\id A_A<\infty$. Then we have the following statements.
\begin{itemize}
  \item [(i)] For every $M^\cdot\in D^-(\Mod A)$, $$R\Gamma_\mfa(M^\cdot)\cong M^\cdot\otimes_A^L R\Gamma_\mfa(A)$$ in $D^-(\Mod A)$.
  \item [(ii)] For every $M^\cdot\in D^-(\Mod A)$, $$R\Gamma_\mfa(M^\cdot)^*\cong \RHom_A(M^\cdot,R\Gamma_\mfa(A)^*)$$ in $D^+(\Mod A^\circ)$.
\end{itemize}
\end{theorem}
\begin{proof} The statement (ii) is similar to \cite[Theorem 5.1]{VdB}. We include the proof for the completeness. Assume that $\Gamma_\mfa$ has finite cohomological dimension (the proof is similar for the case that $A$ has finite injective dimension in both sides). Let $0\to I^0\overset{\delta^0}\to \cdots\overset{\delta^{n-1}}\to I^n\overset{\delta^n}\to \cdots$ be an injective resolution of the $A$-$A$-bimodule $A$. Since $\Gamma_\mfa$ has finite cohomological dimension, we see that $\ker \delta^n$ is $\Gamma_\mfa$ acyclic for some $n$. Let $E^i=I^i$ for $0\leq i<n$ and $E^n=\ker\delta^n$. Denote by $E^\cdot$ the complex $0\to E^0\overset{\delta^0}\to \cdots\overset{\delta^{n-2}}\to E^{n-1}\overset{\delta^n}\to E^n\to0$. Then $R\Gamma_\mfa(A)=\Gamma_\mfa(E^\cdot)$. Let $P^\cdot$ be a projective resolution of $M$. We have $$M^\cdot\otimes_A^LR\Gamma_\mfa(A)\cong P^\cdot\otimes_A \Gamma_\mfa(E^\cdot).$$  On the other hand, consider the complex $P^\cdot\otimes_A E^\cdot$. Since $E^\cdot$ is bounded and is quasi-isomorphic to $A$, $P^\cdot\otimes_A E^\cdot$ is quasi-isomorphic to $M^\cdot$. Moreover, each component of the complex $P^\cdot\otimes E^\cdot$ is $\Gamma_\mfa$-acyclic by Corollary \ref{cor-tens}. Hence $$R\Gamma_\mfa(M^\cdot)\cong \Gamma_\mfa(P^\cdot\otimes_A E^\cdot)\cong P^\cdot\otimes_A\Gamma_\mfa(E^\cdot).$$ Hence (i) follows. Finally, applying the functor $(\ )^*$ to the isomorphism in (i), we obtain Statement (ii).
\end{proof}

\section{Quotient categories} \label{sec-quot}

Let $A$ be a noetherian algebra, and let $\mfa$ be an ideal of $A$. We write $\Tor_\mfa A$ for the full subcategory of $\Mod A$ consisting of $\mfa$-torsion modules. Then $\Tor_\mfa A$ is a localizing subcategory of $\Mod A$. Define $$\QMod_\mfa A=\frac{\Mod A}{\Tor_\mfa A}.$$ The category $\QMod_\mfa A$ is an abelian category. For any module $M\in \Mod A$, we write $\mathcal{M}$ for the corresponding object in $\QMod_\mfa A$ through the natural projection functor $\pi:\Mod A\to \QMod_\mfa A$. Since $\Tor_\mfa A$ is a localizing subcategory of $\Mod A$, the projection functor $\pi$ has a right adjoint functor $\omega:\QMod_\mfa A\longrightarrow \Mod A$. For $M,N\in\Mod A$, the hom-set in $\QMod_\mfa A$ is defined by $$\Hom_{\QMod_\mfa A}(\mathcal{N},\mathcal{M})=\underrightarrow{\lim}\Hom_A(N',M/\Gamma_\mfa(M))$$ where the limit runs over all the submodules $N'$ of $N$ such that $N/N'$ is $\mfa$-torsion. Assume that $N$ is a finitely generated $A$-module. If $N/N'$ is $\mfa$-torsion, then there is an integer $n$ such that $N\mfa^n\subseteq N'$. Hence, in this case,
\begin{equation}\label{eq1}
  \Hom_{\QMod_\mfa A}(\mathcal{N},\mathcal{M})=\underset{n\to\infty}{\lim}\Hom_A(N\mfa^n,M/\Gamma_\mfa(M)).
\end{equation}
We refer the reader to the book \cite{P} for the basic properties of the quotient category $\QMod_\mfa A$.

In the rest of this section, $R$ is a noetherian algebra, and $G$ is a finite group acting on $R$ so that $R$ is a left $G$-module algebra. Let $B=R*G$ be the skew group algebra. Let $\mfr=\mfr(G,R)$ be the $G$-radical of $R$. Set $A=R^G$ and $\mfa=\mfr\cap A$. Then $A$ is a noetherian algebra and $\mfa$ is an ideal of $A$. Let $e=\frac{1}{|G|}\sum_{g\in G} g$. Then $A$ is isomorphic to $eBe$ through the isomorphism $A\to eBe, a\mapsto a\otimes e$. Set $\mfb=\mfr\otimes \kk G$. Since the radical $\mfr$ is stable under the $G$-action, $\mfb$ is an ideal of $B$.

The following facts are easy.
\begin{lemma} \label{lem-tortrans}
\begin{itemize}
  \item [(i)] $\mfb^n=\mfr^n\otimes \kk G$.
  \item [(ii)] If $M$ is a right $B$-module, then $\Gamma_\mfr(M)$ is a right $B$-submodule of $M_B$. Moreover,  $\Gamma_\mfb(M)=\Gamma_\mfr(M)$ as right $B$-modules.
\end{itemize}
\end{lemma}

Since $A$ is a subalgebra of $R$, $\mfa R$ is a right ideal of $R$. In general, $\mfa R$ is not a two-sided ideal of $R$.

\begin{definition} We say that $\mfa$ and $\mfr$ are {\it cofinal}, if (i) $\mfa R=R\mfa$, and (ii) the filtrations $R\supseteq \mfa R\supseteq \mfa^2 R\supseteq\cdots\supseteq \mfa^n R\supseteq\cdots$ and $R\supseteq \mfr \supseteq \mfr^2 \supseteq\cdots\supseteq \mfr^n\supseteq\cdots$ defined by the ideals $\mfa R$ and $\mfr$ of $R$, are cofinal, or equivalently, for each $s$, there is an integer $n$ such that $\mfr^n\subseteq \mfa^sR$.
\end{definition}

\begin{lemma} \label{lem-cofin} Let $R$ and $G$ be as above. Assume that $\mfa$ and $\mfr$ are cofinal.
\begin{itemize}
  \item [(i)] If $N$ is an $\mfa$-torsion right $A$-module, then $\Hom_A(R,N)$ is a $\mfb$-torsion right $B$-module.
  \item [(ii)] If in addition, $\mfa$ has the right the right AR-property, then $\Ext_A^i(R,N)$ is a $\mfb$-torsion right $B$-module for all $i\ge0$.
  \end{itemize}
\end{lemma}
\begin{proof} (i) For every $f\in\Hom_A(R,N)$, by Lemma \ref{lem-tortrans}, we only need to show that $f\mfr^n=0$ for some $n$. Note that $R_A$ is finitely generated. We assume that $R=r_1A+\dots+r_tA$ for some $r_1,\dots,r_t\in R$. Since $N$ is $\mfa$-torsion, there is an integer $s$ such that $f(r_i)\mfa^s=0$ for all $i=1,\dots,t$. Since $\mfa$ and $\mfr$ are cofinal, there is an integer $n$ such that $\mfr^n\subseteq R\mfa^s$. Now for every $x\in R$, we have $(f\mfr^n)(x)= f(\mfr^n x)\subseteq f(\mfr^n)\subseteq f(R\mfa^s)=f(r_1\mfa^s+\cdots+r_t\mfa^s)=f(r_1)\mfa^s+\cdots+f(r_t)\mfa^s=0$. It follows that $f\mfr^n=0$. Hence $\Hom_A(R,N)$ is a $\mfb$-torsion module.

(ii) Take a minimal injective resolution $0\to N\to I^0\to I^1\to\cdots$. Since $\mfa$ has the right AR-property, $I^i$ is $\mfa$-torsion for all $i\ge0$ by Lemma \ref{lem-torinj}. By (i), $\Hom_A(R,I^i)$ is $\mfb$-torsion. Hence $\Ext^i_A(R,N)$ is $\mfb$-torsion for all $i\ge0$.
\end{proof}

Consider the exact functor $-\otimes_B R:\Mod B\longrightarrow \Mod A$. Since $R$ is isomorphic to $Be$ through the isomorphism $R\to Be,r\mapsto r\otimes e$. By Lemma \ref{lem-tortrans}(i), $M\otimes_B R$ is an $\mfa$-torsion module if $M_B$ is a $\mfb$-torsion module since $\mfa=\mfr\cap A$. Hence $-\otimes_BR$ induces an exact functor: $$-\otimes_\mathcal{B}\mathcal{R}:\QMod_\mfb B\longrightarrow \QMod_\mfa A.$$

\begin{theorem}\label{thm-equiv} Let $R$ and $G$ be as above.  Assume that $\mfa$ and $\mfr$ are cofinal, and $\mfa$ has the right AR-property.
\begin{itemize}
  \item [(i)] The functor $\Hom_A(R,-):\Mod A\longrightarrow \Mod B$ induces a functor $$\Hom_\mathcal{A}(\mathcal{R},-):\QMod_\mfa A\longrightarrow \QMod_\mfb B.$$
  \item [(ii)] The functors $-\otimes_\mathcal{B}\mathcal{R}$ and $\Hom_\mathcal{A}(\mathcal{R},-)$ are quasi-inverse to each other.
\end{itemize}
\end{theorem}
\begin{proof} (i) By Lemma \ref{lem-cofin}, we see that $\Hom_A(R,-)$ sends $\mfa$-torsion modules to $\mfb$-torsion modules. Let $f:M_A\to N_A$ be a morphism in $\Mod A$ such that $\ker f$ and $\coker f$ are $\mfa$-torsion modules. We show next that the kernel and cokernel of the morphism $f_*=\Hom_A(R,f):\Hom_A(R,M)\to \Hom_A(R,N)$ are $\mfb$-torsion. By Lemma \ref{lem-cofin}, $\ker f_*=\Hom_A(R,\ker f)$ is $\mfb$-torsion.
Consider the exact sequences:
\begin{equation}\label{ex1}
  0\to\ker f\to M\overset{g}\to K\to0,
\end{equation}
\begin{equation}\label{ex2}
  0\to K\overset{\iota}\to N\to \coker f\to 0,
\end{equation}
where $K=\im f$, $\iota$ is the inclusion map and $g$ is the natural morphism induced by $f$. By the exact sequence (\ref{ex1}), we have an exact sequence $0\to \Hom_A(R,\ker f)\to \Hom_A(R,M)\overset{g_*}\to \Hom_A(R,K)\to \Ext^1_A(R,\ker f)$. Since $\mfa$ has the right AR-property and $\ker f$ is $\mfa$-torsion, $\Ext^1_A(R,\ker f)$ is $\mfb$-torsion by Lemma \ref{lem-cofin}. Hence the quotient module $\Hom_A(R,K)/\im g_*$ is a $\mfb$-torsion $B$-module. By the exact sequence (\ref{ex2}), we see that the quotient module $\Hom_A(R,N)/\im \iota_*$ is a $\mfb$-torsion $B$-module since it is a submodule of the $\mfb$-torsion module $\Hom_A(R,\coker f)$. Note that $f_*=\iota_*g_*$. Thus $\im f_*=\iota_*(\im g_*)$. Hence we have a short exact sequence $$0\longrightarrow \im \iota_*/\im f_*\longrightarrow\Hom_A(R,N)/\im f_*\longrightarrow\Hom_A(R,N)/\im \iota_*\longrightarrow 0.$$ Since $\iota_*$ is injective, we have $\im \iota_*/\im f_*\cong \Hom_A(R,K)/\im g_*$.  As both the first term and the third term are $\mfb$-torsion, the middle term is $\mfb$-torsion as well, which is isomorphic to $\coker f_*$.
Now by the definition of the quotient categories, we see that the functor $\Hom_A(R,-)$ induces a functor between the quotient categories as desired.

(ii) Note that $R$ is isomorphic to $Be$ through the map $r\mapsto r\otimes e$, and $A$ is isomorphic to $eBe$ via the map $a\mapsto a\otimes e$.  in subsequent, we identify the $B$-$A$-bimodule $R$ with $Be$ via the above maps.

For any $N\in \Mod A$, we have
\begin{eqnarray*}
  \Hom_A(R,N)\otimes_BR&\cong& \Hom_{eBe}(Be,N)\otimes_B Be\\
  &\cong& \Hom_{eBe}(Be,N)\otimes_B\Hom_B(eB,B)\\
  &\cong&\Hom_B(eB,\Hom_{eBe}(Be,N))\\
  &\cong&\Hom_{eBe}(eB\otimes_B Be,N)\\
  &\cong& N.
\end{eqnarray*}
Hence we obtain that $\Hom_\mathcal{A}(\mathcal{R},\mathcal{N})\otimes_\mathcal{B}\mathcal{R}$ is naturally isomorphic to $\mathcal{N}$.

For a $B$-module $M$, we have $\Hom_{\mathcal{A}}(\mathcal{R},\mathcal{M}\otimes_\mathcal{B}\mathcal{R})=\pi(\Hom_A(R,M\otimes_BR))$. Similarly we have the following isomorphisms:
\begin{eqnarray*}
 \Hom_A(R,M\otimes_BR)&\cong&\Hom_{eBe}(Be,M\otimes_B Be)\\
 &\cong&\Hom_{eBe}(Be,\Hom_B(eB,M))\\
 &\cong& \Hom_B(Be\otimes_{eBe}eB,M).
\end{eqnarray*}
Consider the exact sequence
\begin{equation}\label{ex3}
0\to L\to Be\otimes_{eBe}eB\overset{m}\to BeB\to0,
\end{equation}
\begin{equation}\label{ex4}
 0\to BeB\overset{\tau}\to B\to B/BeB\to0,
\end{equation}
where $m$ is the multiplication of the algebra $B$, $\tau$ is the inclusion map and $L$ is the kernel of $m$.

Since $\mfb\subseteq BeB$, $B/BeB$ is $\mfb$-torsion on both sides. Applying the functors $-\otimes_B Be$ and $eB\otimes_B-$ to the exact sequence (\ref{ex3}), we obtain that $Le=0$ and $eL=0$. Hence $L$ is $\mfb$-torsion both as a right $B$-module and as a left $B$-module. Since $B$ is noetherian and $Be$ is finitely generated as an $eBe$-module, both $B/BeB$ and $L$ are finitely generated as right $B$-modules. It follows from the exact sequences (\ref{ex3}),  (\ref{ex4}) and Lemma \ref{lem-Btor} below, that the kernels and cokernels of both  maps $\Hom_B(\tau,M):M\longrightarrow \Hom_B(BeB,M)$ and  $\Hom_B(m,M):\Hom_B(BeB,M)\longrightarrow\Hom_B(Be\otimes_{eBe} eB, M)$ are $\mfb$-torsion. Therefore, we obtain a natural isomorphism in $\QMod_\mfb B$: $$\pi(\Hom_B(Be\otimes_{eBe}eB,M))\cong \mathcal{M}.$$ Hence we have a natural isomorphism $\Hom_{\mathcal{A}}(\mathcal{R},\mathcal{M}\otimes_\mathcal{B}\mathcal{R})\cong \mathcal{M}$.
\end{proof}

\begin{lemma} \label{lem-Btor} Assume that $X$  is a $B$-$B$-bimodule which is finitely generated as a right $B$-module. If $X$ is $\mfb$-torsion as a left $B$-module, then $\Ext^j_B(X,M)$ is a $\mfb$-torsion right $B$-module for all $i\ge0$ and any right $B$-module $M$.
\end{lemma}
\begin{proof} Since $X$ is finitely generated as a right $B$-module, we have $X=x_1B+\cdots+x_tB$ for some $x_1,\dots,x_t\in X$.  As $X$ is $\mfb$-torsion as a left $B$-module, there is an integer $n$ such that $\mfb^mx_i=0$ for all $i=1,\dots,t$. For any element $f\in \Hom_B(X,M)$, we have $(f\mfb^n)(x_i)=f(\mfb^n x_i)=0$ for $i=1,\dots,t$. Hence $\Hom_B(X,M)$ is $\mfb$-torsion. Now take an injective resolution $0\to M\to I^0\to I^2\to\cdots$. We see that $\Hom_B(X,I^j)$ is $\mfb$-torsion for all $j\ge0$. Hence $\Ext^j_B(X,M)$ is $\mfb$-torsion for all $j\ge0$.
\end{proof}

It is known that any ideal of a noetherian commutative algebra has the AR-property (cf. \cite[Theorem 8.5]{M}). The condition that $\mfa$ and $\mfr$ are cofinal in the above theorem is necessary when the algebra $R$ is commutative.

\begin{theorem} \label{thm-commeq} Assume that $R$ is a noetherian commutative algebra and $G$ is a finite group acting on $R$ by automorphisms. Then the following are equivalent.
\begin{itemize}
  \item [(i)] The functor $-\otimes_\mathcal{B}\mathcal{R}:\QMod_\mfb B\longrightarrow \QMod_\mfa A$ is an equivalence of abelian categories.
  \item [(ii)] The ideal $\mfa$ of $A$ and the ideal $\mfr$ of $R$ are cofinal.
\end{itemize}
\end{theorem}
\begin{proof} That (ii) implies (i) has been proved in Theorem \ref{thm-equiv}. Now assume that (i) holds. Write $\mathcal{F}$ for $-\otimes_\mathcal{B}\mathcal{R}$. For $s\ge0$, let $M=R/\mfa^sR$. Note that $\mfa^sR$ is a right $B$-submodule of $R$. Hence $M$ is a right $B$-module. Then $\mathcal{F}(\mathcal{M})=\pi(M\otimes_B R)=0$ since $(M\otimes_BR)\mfa^s=M\otimes_BR\mfa^s=M\mfa^s\otimes_BR=0$. Therefore $\mathcal{M}=0$, equivalently, $M$ is an $\mfr$-torsion module. Since $M$ is finitely generated, there is an integer $n$ such that $M\mfr^n=0$, that is, $\mfr^n\subseteq \mfa^sR$.
\end{proof}

\begin{example} Let $R=\kk[x,y]$ and let $G=\{1,g\}$ where $g$ is the automorphism of $R$ defined by $g(x)=-x$ and $g(y)=-y$. It is easy to see that $(1,x)$ and $(x,1)$ are pertinent, hence $x\in \mfr(R,G)$. Similarly, $y\in \mfr(R,G)$. Therefore $\mfr:=\mfr(R,G)=Rx+Ry$. The invariant subalgebra $A:=R^G$ is the subalgebra of $R$ consisting of elements of even degrees (here, we view $x$ and $y$ of degree 1). Then $\mfa=\mfr\cap A=Ax^2+Ay^2+Axy$. Now it is easy to see that $\mfa$ and $\mfr$ are cofinal.
\end{example}

\begin{example}\label{ex5} Let $R=\kk_{-1}[x,y,z]$, and let $G=\{1,g\}$ where $g$ is the automorphism of $R$ defined by $g(x)=x$, $g(y)=-y$ and $g(z)=-z$. Then the invariant subalgebra $A=\kk[x,y^2,z^2,yz]$. As in the above example, $y,z$ are in the $G$-radical $\mfr=\mfr(R,G)$. Let $\mathfrak{i}=Ry+Rz$. Then $\mathfrak{i}\subseteq\mfr$. Note that $\mathfrak{i}$ is closed under $G$-action, hence $R/\mathfrak{i}$ is a $G$-module algebra. One sees that the induced $G$-action on $\kk[x]\cong R/\mathfrak{i}$ is trivial. Hence by Lemma \ref{lem-min}, $\mathfrak{i}=\mfr$. Then $\mfa=\mfr\cap A=Ay^2+Az^2+Ayz$. Since $A$ is commutative, $\mfa$ has the (right) AR-property. Moreover, one sees that $\mfa$ and $\mfr$ are cofinal.
\end{example}

\begin{example} Let $R=\kk_{-1}[x,y,z]$, and let $G=\{1,g,g^2\}$ where $g$ is the automorphism of $R$ defined by $g(y)=\omega y$ and $g(z)=\omega^2 z$ in which $\omega$ is a third primitive root of unity. Then it is easy to see that the invariant subaglebra $A$ is generated by $x,y^3,z^3,yz$. We next compute the $G$-radical of $R$. An easy computation shows $(y^2,y,1)\overset{G}\sim(1,y,y^2)$, $(z^2,z,1)\overset{G}\sim(1,z,z^2)$ and $(yz,-z,y)\overset{G}\sim(1,y,z)$. Hence $y^2,z^2,yz\in \mfr$. Let $\mathfrak{i}$ be the ideal of $R$ generated by $y^2,z^2,yz$. Then $\mathfrak{i}\subseteq \mfr$. Consider the quotient algebra $R/\mathfrak{i}$. Note that $R/\mathfrak{i}$ is also a $G$-module algebra and the invariant subalgebra is isomorphic to $\kk[x]$. As a left $\kk[x]$-module, $R/\mathfrak{i}$ is isomorphic to $\kk[x]\oplus \kk[x]y\oplus \kk[x]z$. By Proposition \ref{prop-ss}, $R/\mathfrak{i}$ is $G$-semisimple. By Lemma \ref{lem-min}, $\mfr=\mathfrak{i}$. Then $\mfa=\mfr\cap A=A y^3+Az^3+Ayz$. Since $y^3,z^3,yz$ are normal elements in $A$, $\mfa$ has the right AR-property (cf. \cite[Proposition 4.2.6]{MR}). Also, it is easy to see that $\mfa$ and $\mfr$ are cofinal.
\end{example}

\section{Relative Cohen-Macaulay algebras}\label{sec6}

Let $R$ be a noetherian algebra, and let $G$ be a finite group acting on $R$ by automorphims. Let $A=R^G$ be the invariant subalgebra. As before, $\mfr=\mfr(R,G)$ is the $G$-radical of $R$ and $\mfa=\mfr\cap A$. Set $B=R*G$ and $\mfb=\mfr\otimes\kk G$.

\begin{setup}\label{set} In this section, we assume that  the following conditions hold:
\begin{itemize}
  \item [(i)] $\mfa$ has the right AR-property;
  \item [(ii)] $\mfa$ and $\mfr$ are cofinal.
\end{itemize}
\end{setup}

\begin{remark}\label{rem1}  If $\mfa=a_1A+\cdots+a_sA$, where $a_1,\dots,a_s$ are normal elements both in $A$ and in $R$, then Setup \ref{set}(i) is automatically satisfied.
\end{remark}

\begin{lemma}\label{lem-ARpro} Both ideals $\mfr$ and $\mfb$ have the right AR-property.
\end{lemma}
\begin{proof} Let $M$ be a finitely generated right $R$-module, and let $N$ be a $B$-submodule of $M$. Since $R$ is finitely generated as a right $A$-module, both $M$ and $N$ are finitely generated as right $A$-modules. By Setup \ref{set}(ii), $\mfa R=R\mfa$. Let $I=R\mfa$. Then $I^n=R\mfa^n$ for $n>0$. Hence $MI^n=M\mfa^n$. Since $\mfa$ has the right AR-property, $MI^n\cap N=M\mfa^n\cap N\subseteq N\mfa\subseteq NI$ for some $n>0$. Therefore $I$ has the right AR-property. Since the filtrations $I\supseteq I^2\supseteq\cdots\supseteq I^n\supseteq\cdots$ and $\mfr\supseteq\mfr^2\supseteq\cdots\supseteq\mfr^n\supseteq\cdots$ are cofinal, it follows that $\mfr$ has the right AR-property.

Similarly, we may show that $\mfb$ has the right AR-property by noticing that $K\mfb^n=K\mfr^n$ for every finitely generated right $B$-module $K$.
\end{proof}

Let $M$ be a right $A$-module, and let
\begin{equation}\label{eq-res}
  0\to M\to I^0\overset{\delta^0}\to I^{1}\overset{\delta^1}\to\cdots\to I^i\overset{\delta^j}\to\cdots
\end{equation}
be an injective resolution of $M$. Let $T^i=\Gamma_\mfa(I^i)$ for each $i$. By Lemma \ref{lem-torinj}, we see that $T^i$ is an injective $\mfa$-torsion module. Hence, for each $i\ge0$, we have a decomposition $I^i=T^i\oplus E^i$ where $E^i$ is an $\mfa$-torsion-free injective module. The differential $\delta^i$ has a decomposition $\delta^i=\delta^i_E+\delta^i_T+f^i$ where $\delta_E^i:E^i\longrightarrow E^{i+1}$, $\delta_T^i::T^i\longrightarrow T^{i+1}$ and $f^i: E^i\longrightarrow T^{i+1}$. Since $\delta^{i+1}\delta^i=0$, we have $\delta_E^{i+1}\delta_E^i=0$, $\delta_T^{i+1}\delta_T^i=0$ and $f^{i+1}\delta^i_T+\delta^{i+1}_Ef^i=0$. Let $E^\cdot$ (resp. $T^\cdot$) denote the complex with differential $\delta^\cdot_E$ (resp. $\delta_T^\cdot$). Then
\begin{equation}\label{eq-resdec}
  f^\cdot:E^\cdot[-1]\longrightarrow T^\cdot
\end{equation}
is a morphism of complex, and the injective resolution $I^\cdot=cone(f^\cdot)$.

\begin{lemma}\label{lem-locoh} With the notions as above, $R^i\Gamma_\mfa(M)=H^i(T^\cdot)$ for all $i\ge0$, where $H^i(-)$ is the $i$-th cohomology of the complex.
\end{lemma}

Recall from Section \ref{sec-quot} that $\pi$ is the projection functor $\Mod A\longrightarrow \QMod_\mfa A$, and its right adjoint functor is $\omega:\QMod_\mfa A\longrightarrow\Mod A$. For $M\in \Mod A$, we write $\mathcal{M}$ for the object $\pi(M)$. From the decomposition of the injective resolution (\ref{eq-res}) of $M$ above, we obtain an injective resolution of $\mathcal{M}$ in $\QMod_\mfa A$ (cf. \cite[Corollary 5.4]{P}) $$0\to \mathcal{M}\to \mathcal{E}^0\to\mathcal{E}^1\to\cdots\to\mathcal{E}^i\to\cdots$$ where $\mathcal{E}^i=\pi(E^i)$ for all $i\ge0$.

\begin{lemma} \label{lem-secfun} With the notions as above, $\Ext^i_{\QMod_\mfa A}(\mathcal{A},\mathcal{M})\cong H^i(E^\cdot)$ for all $i\ge0$.
\end{lemma}
\begin{proof} Applying the functor $\Hom_{\QMod_\mfa A}(\mathcal{A},-)$ to the injective resolution $\mathcal{E}^\cdot$, we obtain a complex $\Hom_{\QMod_\mfa A}(\mathcal{A},\mathcal{E}^\cdot)$. Since $E^i$ is $\mfa$-torsion-free and injective for all $i\ge0$, we have a natural isomorphism $\omega(\pi(E^i))\cong E^i$. Hence we have natural isomorphisms $\Hom_{\QMod_\mfa A}(\mathcal{A},\mathcal{E}^i)\cong \Hom_A(A,\omega(\pi(E^i)))\cong E^i$. Therefore we have an isomorphism of complexes $\Hom_{\QMod_\mfa A}(\mathcal{A},\mathcal{E}^\cdot)\cong E^\cdot$.
\end{proof}

\begin{proposition}\label{prop-finloc} If $R$ has finite global dimension, then the torsion functor $\Gamma_\mfa$ has finite cohomological dimension.
\end{proposition}
\begin{proof} Assume that the global dimension of $R$ is $d$. For $N\in\Mod B$, let $0\to N\to Q^0\to \cdots\to Q^d\to0$ be an injective resolution. By Lemma \ref{lem-ARpro}, $\mfb$ has the right AR-property. Hence the injective envelope of a $\mfb$-torsion $B$-module is still $\mfb$-torsion by Lemma \ref{lem-torinj}. By \cite[Corollary 5.4]{P}, $0\to\pi(N)\to\pi(Q^0)\to\cdots\to\pi(Q^d)\to0$ is an injective resolution in $\QMod_\mfb B$. Hence the global dimension of $\QMod_\mfb B$ is not larger than $d$. Setup \ref{set} implies that the conditions in Theorem \ref{thm-equiv} are satisfied. Hence the abelian categories $\QMod_\mfa A$ and $\QMod_\mfb B$ are equivalent. Therefore, the global dimension $\QMod_\mfa A$ is not larger than $d$.

For $M\in\Mod A$, let $$0\to M\to I^0\overset{\delta^0}\to I^{1}\overset{\delta^1}\to\cdots\to I^i\overset{\delta^j}\to\cdots$$ be an injective resolution. As we have seen that $I^\cdot=cone(f^\cdot)$ where $f^\cdot:E^\cdot\to T^\cdot$ is the morphism given in (\ref{eq-resdec}). Since the global dimension of $\QMod_\mfa A$ is not larger than $d$, we have $\Ext_{\QMod_\mfa A}^i(\pi(A),\pi(M))=0$ for all $i>d$. Then Lemma \ref{lem-secfun} implies that $H^i(E^
\cdot)=0$ for $i>d$. From the exact sequence $0\to T^\cdot\to I^\cdot\to E^\cdot\to0$ of complexes, we obtain that $H^i(T^\cdot)=0$ for $i>d+1$. Now $R^i\Gamma_\mfa(M)=H^i(\underset{n\to\infty}\lim\Hom_A(A/\mfa^n, I^\cdot))=H^i(T^\cdot)=0$ for $i>d+1$. Therefore $\Gamma_\mfa$ has finite cohomological dimension.
\end{proof}

Note that the global dimension of the invariant subaglebra $A$ is often infinite even if the global dimension of $R$ is finite. For example, it is well known that a nontrivial finite subgroup of the special linear group $SL_n(\kk)$ acting on the polynomial algebra $\kk[x_1,\dots,x_n]$ yields a Gorenstein invariant subalgebra with infinite global dimension. More general, it is known that a finite group acting on an Artin-Schelter regular algebra with trivial homological determinant gives rise to an Artin-Schelter Gorenstein invariant subalgebra (cf. \cite{JZ}). Hence it is reasonable to consider the Cohen-Macaulay modules over the invariant subalgebra.
Similar to the concept introduced in (cf. \cite{Z,CH}), we make the following definition.

\begin{definition} We say that the invariant subalgebra $A=R^G$ is (right) {\it $\mfa$-Cohen-Macaulay of dimension $d$} if $R^i\Gamma_\mfa(A)=0$ for $i\neq d$.
\end{definition}

\begin{example} Recall from Example \ref{ex5} that the ideals $\mfa$ and $\mfr$ satisfy the conditions in Setup \ref{set}. Note that $y^2,z^2,yz$ are normal elements in $R$. The invariant subalgebra $A=\kk[x,y^2,z^2,yz]=\kk[x]\otimes \Lambda$ where $\Lambda=\kk[u,v,w]/(uv-w^2)$. Under this isomorphism, $\mfa\cong \kk[x]\otimes \mathfrak{m}$ where $\mathfrak{m}=\Lambda u+\Lambda v+\Lambda w$. It is well known that $\Lambda$ has injective dimension $2$.  Moreover,  $R^i\Gamma_\mathfrak{m}(\Lambda)=0$ for $i=0,1$ and $R^2\Gamma_\mathfrak{m}(\Lambda)=E(\kk)$, where $E(\kk)$ is the injective envelope of the trivial module $\kk$. Thus, we have: $$R^i\Gamma_\mfa(A)\cong \underset{n\to\infty}\lim\Ext_{\kk[x]\otimes \Lambda}^i(\kk[x]\otimes \Lambda/\mathfrak{m}^n,\kk[x]\otimes\Lambda)\cong\kk[x]\otimes\underset{n\to\infty}\lim\Ext_{\Lambda}^i(\Lambda/\mathfrak{m}^n,\Lambda).
$$ Hence $R^i\Gamma_\mfa(A)=0$ for $i\neq 2$ and $R^2\Gamma_\mfa(A)\cong \kk[x]\otimes E(\kk)$.
\end{example}

\begin{theorem}\label{thm-finproinj} Assume that $R$ has finite global dimension and $A$ is $\mfa$-Cohen-Macaulay of dimension $d$. Set $D=R^d\Gamma_\mfa(A)$.
For $M\in\Mod A$, $$R^i\Gamma_\mfa(M)\cong \Tor_{d-i}^A(M,D)$$ for all $i\ge0$.
\end{theorem}
\begin{proof}  By Proposition \ref{prop-finloc} and Theorem \ref{thm-locdual}, we have $R\Gamma_\mfa(M)\cong M\otimes_A^LR\Gamma_\mfa(A)\cong M\otimes_A^L D[-d]$. Taking the cohomology of the complexes, we obtain the desired isomorphisms.
\end{proof}

\section{Extension groups in the quotient categories}\label{sec7}

Let $S$ be a noetherian algebra, and let $\mathfrak{s}$ be an ideal of $S$. Assume that $\mathfrak{s}$ has the right AR-property. We have the following computation of extension groups in the quotient category.

\begin{lemma} \label{lem-qext} If $N$ is a finitely generated right $S$-module and $M$ is a right $S$-module, then $\Ext^i_{\QMod_{\mathfrak{s}} S}(\mathcal{N},\mathcal{M})=\underset{n\to\infty}\lim\Ext^i_S(N\mathfrak{s}^n,M)$ for $i\ge0$.
\end{lemma}
\begin{proof} Let $T$ be an injective $\mathfrak{s}$-torsion $S$-module. We claim $\underset{n\to\infty}\lim\Hom_S(N\mathfrak{s}^n,T)=0$. Indeed, applying the functor $\Hom_S(-,T)$ to the exact sequence $0\to N\mathfrak{s}^n\to N\to N/N\mathfrak{s}^n\to0$, we obtain the exact sequence $0\to \Hom_S(N/N\mathfrak{s}^n,E)\to\Hom_S(N,T)\to \Hom_S(N\mathfrak{s}^n,T)\to0$. Taking the direct limit we have $$0\to \underset{n\to\infty}\lim\Hom_S(N/N\mathfrak{s}^n,E)\to\Hom_S(N,T)\to \underset{n\to\infty}\lim\Hom_S(N\mathfrak{s}^n,T)\to0.$$ Since $N$ is finitely generated, for any $S$-module homomorphism $f:N\to T$, there is an integer $n$ such that $f(N\mathfrak{s}^n)=f(N)\mathfrak{s}^n=0$. Hence $N\mathfrak{s}^n\subseteq \ker f$. Therefore, $f:N\to T$ factors though $N/N\mathfrak{s}^n$. Thus the morphism $\underset{n\to\infty}\lim\Hom_S(N/N\mathfrak{s}^n,T)\to\Hom_S(N,T)$ in the exact sequence above is an epimorphism. So, the claim holds.

We continue to prove the lemma.
Take a minimal injective resolution of the right $S$-module $M$ as follows:
\begin{equation}\label{eq-1inj}
  0\to M\to I^0\to I^2\to\cdots\to I^k\to\cdots.
\end{equation}
Since $\mathfrak{s}$ has the right AR-property, the injective module $I^i$ $(i\ge0)$ has a decomposition $I^i=T^i\oplus E^k$ with $T^i$ a $\mathfrak{s}$-torsion submodule and $E^i$ a $\mathfrak{s}$-torsion free submodule. Applying the projection functor $\pi:\Mod S\longrightarrow\QMod_\mathfrak{s} S$ to the projective resolution (\ref{eq-inj}), we obtain the following exact sequence
\begin{equation}\label{eq-1injq}
  0\to \mathcal{M}\to \mathcal{E}^0\to \mathcal{E}^2\to\cdots\to \mathcal{E}^i\to\cdots.
\end{equation}
Since $E^i$ is $\mathfrak{s}$-torsion free, $\mathcal{E}^i$ is injective in $\QMod_\mathfrak{s} S$ for all $i\ge0$. Hence the exact sequence (\ref{eq-1injq}) provides an injective resolution of $\mathcal{M}$ in $\QMod_\mathfrak{s} S$.
 Applying the functor $\Hom_{\QMod_\mathfrak{s} S}(\mathcal{N},-)$ to (\ref{eq-1injq}), we have the following complex
\begin{equation}\label{eq-1coh2}
  0\to\Hom_{\QMod_\mathfrak{s} S}(\mathcal{N},\mathcal{E}^0)\to\cdots\to \Hom_{\QMod_\mathfrak{s} S}(\mathcal{N},\mathcal{E}^i)\to\cdots.
\end{equation}
 By Equation (\ref{eq1}), the complex (\ref{eq-1coh2}) is isomorphic to the following complex
\begin{equation}\label{eq-1coh4}
  0\to\underset{n\to\infty}\lim\Hom_ S(N\mathfrak{s}^n,E^0)\to\cdots\to \underset{n\to\infty}\lim\Hom_S(N\mathfrak{s}^n,E^i)\to\cdots.
\end{equation}
By the above claim, (\ref{eq-1coh4}) is isomorphic to
\begin{equation}\label{eq-1coh5}
  0\to\underset{n\to\infty}\lim\Hom_ S(N\mathfrak{s}^n,E^0\oplus T^0)\to\cdots\to \underset{n\to\infty}\lim\Hom_S(N\mathfrak{s}^n,E^i\oplus T^i)\to\cdots,
\end{equation}
which is equivalent to
\begin{equation}\label{eq-1coh6}
  0\to\underset{n\to\infty}\lim\Hom_ S(N\mathfrak{s}^n,I^0)\to\cdots\to \underset{n\to\infty}\lim\Hom_S(N\mathfrak{s}^n,I^i)\to\cdots.
\end{equation}
By (\ref{eq-1inj}), the $i$-th cohomology of the complex (\ref{eq-1coh6}) is $\underset{n\to\infty}\lim\Ext_S^i(N\mathfrak{s}^n,M)$.
\end{proof}

Throughout the rest of this section, we let $G$ be a finite group, and let $R$ be a noetherian left $G$-module algebra. As before, write $B=R* G$, $A=R^G$, $\mfr=\mfr(R,G)$, $\mfb=\mfr\otimes \kk G$ and $\mfa=A\cap \mfr$. Assume that the radical $\mfr$ has the right AR-property. It follows from the proof of Lemma \ref{lem-ARpro} that $\mfb$ has the right AR-property as well.

Let $N$ and $M$ be $B$-modules. Assume that $N$ is finitely generated. Note that we may view $N$ and $M$ as right $R$-modules. Then there is a right $G$-action $\leftharpoonup$ on $\Hom_R(N,M)$ defined by
\begin{equation}\label{eq-gact}
  (f\leftharpoonup g)(n)=f(ng^{-1})g, \text{ for all } g\in G, f\in\Hom_R(N,M), n\in N.
\end{equation}
With this right $G$-action on $\Hom_R(N,M)$, we have $$\Hom_B(N,M)=\Hom_R(N,M)^G,$$ where the right hand side is the $G$-invariant subspace. Consider the Hom-sets in the quotient categories $\QMod_\mfr R$ and $\QMod_\mfb B$. For any $n\ge0$, $N\mfr^n$ is a right $B$-submodule, and hence $\Hom_R(N\mfr^n,M)$ has a right $G$-action. It is easy to see that the direct limit system $\underset{n\rightarrow\infty}\lim\Hom_R(N\mfr^n,M)$ is compatible with the right $G$-actions. By Lemma \ref{lem-qext}, $\Hom_{\QMod_\mfr R}(\mathcal{N},\mathcal{M})=\underset{n\rightarrow\infty}\lim\Hom_R(N\mfr^n,M)$ has a right $G$-action. Moreover, since $N\mfr^n=N\mfb^n$, we have
\begin{eqnarray*}
\Hom_{\QMod_\mfb B}(\mathcal{N},\mathcal{M})&=&\underset{n\rightarrow\infty}\lim\Hom_B(N\mfb^n,M)\\
&=&\underset{n\rightarrow\infty}\lim\Hom_R(N\mfr^n,M)^G\\
&=&(\underset{n\rightarrow\infty}\lim\Hom_R(N\mfr^n,M))^G\\
&=&\Hom_{\QMod R}(\mathcal{N},\mathcal{M})^G,
\end{eqnarray*}
where the last equality holds because $G$ is a finite group. Summarizing the above arguments, we obtain the the following result.

\begin{lemma}\label{lem-gact} Let $N$ and $M$ be right $B$-modules. Assume that $N$ is finitely generated. Then there is a right $G$-action on $\Hom_{\QMod_\mfr R}(\mathcal{N},\mathcal{M})$ induced from (\ref{eq-gact}). Moreover, we have $$\Hom_{\QMod_\mfb B}(\mathcal{N},\mathcal{M})\cong \Hom_{\QMod R}(\mathcal{N},\mathcal{M})^G.$$
\end{lemma}

If $M'$ is another $B$-module and $f:M\to M'$ is a $B$-module homomorphism, one sees that $f$ is compatible with the right $G$-module structures on $\Hom_R(N,M)$ and $\Hom_R(N,M')$. Moreover, $f$ is compatible with the direct limit systems in the above narratives. Hence $f$ induces a $G$-module homomorphism $$\Hom_{\QMod_\mfr R}(\mathcal{N},\mathcal{M})\longrightarrow \Hom_{\QMod_\mfr R}(\mathcal{N},\mathcal{M}').$$

Next we show that $\Ext^i_{\QMod R}(\mathcal{N},\mathcal{M})$ has a right $G$-action as well, and the above isomorphism may be extended to the extension groups.
Take a minimal injective resolution of the right $B$-module $M$ as follows:
\begin{equation}\label{eq-inj}
  0\to M\to I^0\to I^2\to\cdots\to I^k\to\cdots.
\end{equation}
Since $\mfb$ has the right AR-property, as we have seen in the proof of Lemma \ref{lem-qext}, the injective module $I^i$ $(i\ge0)$ has a decomposition $I^i=T^i\oplus E^k$ with $T^i$ a $\mfb$-torsion submodule and $E^i$ a $\mfb$-torsion free submodule. Then we have an injective resolution of $\mathcal{M}$ in $\QMod_\mfb B$:
\begin{equation}\label{eq-injq}
  0\to \mathcal{M}\to \mathcal{E}^0\to \mathcal{E}^2\to\cdots\to \mathcal{E}^i\to\cdots.
\end{equation}

\begin{proposition}\label{prop-qext} Let $N$ and $M$ be right $B$-modules. Assume that $N$ is finitely generated. The following statements hold:
\begin{itemize}
  \item [(i)] $\Ext^i_{\QMod_\mfr R}(\mathcal{N},\mathcal{M})$ has a right $G$-action for all $i\ge0$.
  \item [(ii)] $\Ext^i_{\QMod_\mfr R}(\mathcal{N},\mathcal{M})^G\cong \Ext^i_{\QMod_\mfb B}(\mathcal{N},\mathcal{M})$ for all $i\ge0$.
  \end{itemize}
\end{proposition}
\begin{proof} (i) Note that a $B$-module is injective if and only if it is injective as an $R$-module (cf. \cite[Proposition 2.6]{HVZ2}). Then we may view (\ref{eq-inj}) as an injective resolution of the $R$-module $M$ in $\Mod R$. Since a $B$-module is $\mfb$-torsion free if and only if it is $\mfr$-torsion free as an $R$-module, the exact sequence (\ref{eq-injq}) is also an injective resolution of $\mathcal{M}$ in $\QMod_\mfr R$.
Applying the functor $\Hom_{\QMod_\mfr R}(\mathcal{N},-)$ to the injective resolution (\ref{eq-injq}), we obtain a complex of right $G$-modules
\begin{equation}\label{eq-coh}
  0\to\Hom_{\QMod_\mfr R}(\mathcal{N},\mathcal{E}^0)\to\cdots\to \Hom_{\QMod_\mfr R}(\mathcal{N},\mathcal{E}^i)\to\cdots.
\end{equation}
Taking the cohomology of the above complex, we obtain $\Ext^i_{\QMod_\mfr R}(\mathcal{N},\mathcal{M})$ which inherits the right $G$-module structure on $\Hom_{\QMod_\mfr R}(\mathcal{N},\mathcal{E}^i)$. The statement (i) follows.\\

(ii) Applying the functor $\Hom_{\QMod_\mfb B}(\mathcal{N},-)$ to (\ref{eq-injq}), we obtain the following complex
\begin{equation}\label{eq-coh2}
  0\to\Hom_{\QMod_\mfb B}(\mathcal{N},\mathcal{E}^0)\to\cdots\to \Hom_{\QMod_\mfb B}(\mathcal{N},\mathcal{E}^i)\to\cdots.
\end{equation}
By Lemma \ref{lem-gact}, the complex (\ref{eq-coh2}) is isomorphic to the following complex
\begin{equation}\label{eq-coh3}
  0\to\Hom_{\QMod_\mfr R}(\mathcal{N},\mathcal{E}^0)^G\to\cdots\to \Hom_{\QMod_\mfr R}(\mathcal{N},\mathcal{E}^i)^G\to\cdots.
\end{equation}
Taking the cohomology of the complex (\ref{eq-coh3}), we obtain the desired isomorphisms in (ii).
\end{proof}

Recall that a noetherian algebra $S$ with finite GK-dimension is called a {\it right GKdim-Cohen-Macaulay algebra} if for any finitely generated right $S$-module $K$, $\GKdim(K)+j_S(K)=\GKdim(S)$, where $j_S(M)=\min\{i|\Ext_S^i(M,S)\neq 0\}$.

Now we return to our noetherian algebra $R$ with a $G$-action. Assume that the GK-dimension on right $R$-modules is exact, that is, if $0\to N\to M\to K\to0$ is an exact sequence of finitely generated right $B$-modules, then $\GKdim(M)=\max\{\GKdim(N),\GKdim(K)\}$. For instance, if $R$ is $\mathbb{Z}$-graded or filtered with an ascending locally finite filtration, then the GK-dimension is exact (cf. \cite{KL}). For any integer $n\ge1$, we have $\GKdim(R/\mfr)=\GKdim(R/\mfr^n)$. If furthermore, $R$ is right GKdim-Cohen-Macaulay, then $j_R(R/\mfr^n)=\GKdim(R)-\GKdim(R/\mfr^n)=\Pty(R,G)$. We next show that $\depth_{\mfr}(R)$ is often equal to the pertinency $\Pty(R,G)$.

\begin{proposition}\label{prop3-2} Let $G$ be a finite group, and let $R$ be a noetherian $G$-module algebra with finite GK-dimension. Assume that $R$ is GKdim-Cohen-Macaulay and that the GK-dimension is exact on right $R$-modules. Then $\depth_\mfr(R)=\Pty(R,G)$.
\end{proposition}
\begin{proof} Let $p=\Pty(R,G)$. As we know, $j_R(R/\mfr^n)=p$ for $n>0$. Then $\Ext^i_R(R/\mfr^n,R)=0$ for all $i<p$. Hence $R^i\Gamma_\mfr(R)=\underset{n\to\infty}\lim\Ext^i_R(R/\mfr^n,R)=0$ for all $i<\Pty(R,G)$. Let $0\to R\to I^0\overset{\delta^0}\to I^1\overset{\delta^1}\to\cdots\to I^i\overset{\delta^i}\to\cdots$ be a minimal injective resolution of $R$. By Lemma \ref{lem-injres}, $I^i$ is $\mfr$-torsion free for all $i<p$. Since $j_R(R/\mfr^n)=p$ for all $n>0$, $\Gamma_\mfr(I^p)\neq0$. Then $\Gamma_\mfr(I^p)\cap \ker \delta^p\neq0$ since $I^\cdot$ is a minimal injective resolution. Since $I^i$ is $\mfr$-torsion free for all $i<p$, it follows that $R^p\Gamma_\mfr(R)=\Gamma_\mfr(I^p)\cap\ker\delta^p\neq0$. Hence $\depth_\mfr(R)=p$.
\end{proof}

\begin{remark} In Proposition \ref{prop3-2}, we don't need the assumption that $\mfr$ has the right AR-property.
\end{remark}

The above proposition shows that $R^i\Gamma_\mfr(R)=0$ for $i<\Pty(R,G)$. Sometimes the local homology of $R$ will be concentrated in degree $p=\Pty(R,G)$ as shown in the next example.

\begin{example} Let $R=\kk[x,y,z]$ and let $\sigma$ be the automorphism of $R$ defined by $\sigma(x)=x$, $\sigma(y)=-y$ and $\sigma(z)=-z$. Let $G=\langle \sigma\rangle$. Then $\mfr=\mfr(R,G)=Ry+Rz$, and $\Pty(R,G)=2$. Let $\Lambda=\kk[y,z]$, and $\mathfrak{m}=(y,z)$ the maximal ideal generated by $y,z$. For any $n\ge1$, $R/\mfr^n\cong \kk[x]\otimes \Lambda/\mathfrak{m}^n$. Hence
\begin{eqnarray*}
 R^i\Gamma_\mfr(R)&=&\underset{n\to\infty}\lim\Ext^i_R(R/\mfr^n,R)\\
 &=&\underset{n\to\infty}\lim\Ext^i_{\kk[x]\otimes \Lambda}(\kk[x]\otimes\Lambda/\mathfrak{m}^n,\kk[x]\otimes\Lambda)\\
 &\cong& \kk[x]\otimes\underset{n\to\infty}\lim\Ext^i_{ \Lambda}(\Lambda/\mathfrak{m}^n,\Lambda).
\end{eqnarray*}
Now $\underset{n\to\infty}\lim\Ext^i_{ \Lambda}(\Lambda/\mathfrak{m}^n,\Lambda)=0$ for $i\neq 2$, and $\underset{n\to\infty}\lim\Ext^2_{ \Lambda}(\Lambda/\mathfrak{m}^n,\Lambda)\cong E(\kk)$, where $E(\kk)$ is the injective envelope of $\kk$ as a $\Lambda$-module.
\end{example}

\begin{definition} \label{def-gcm} Let $G$ be a finite group, and let $R$ be a noetherian $G$-module algebra with finite GK-dimension. If $R^i\Gamma_\mfr(R)=0$ for all $i\neq \Pty(R,G)$, then we call $R$ a {\it $G$-Cohen-Macaulay} algebra.
\end{definition}

We may now prove the main result of this section. Recall that $B=R*G$, $A=R^G$, $\mfr=\mfr(R,G)$, $\mfb=\mfr\otimes \kk G$ and $\mfa=A\cap \mfr$.

\begin{theorem}\label{thm-gcm} Let $G$ be a finite group, and  $R$ a noetherian left $G$-module algebra with finite GK-dimension. Assume that $R$ satisfies Setup \ref{set}. If $R$ is $G$-Cohen-Macaulay, then $A$ is $\mfa$-Cohen-Macaulay of dimension $\Pty(R,G)$.
\end{theorem}
\begin{proof} By Theorem \ref{thm-equiv}, we have an equivalence of abelian categories $$-\otimes_{\mathcal{B}}\mathcal{R}:\QMod_\mfa A\longrightarrow \QMod_\mfb B.$$ Under this equivalence, the object $\mathcal{R}\in \QMod_\mfb B$ corresponds to $\mathcal{A}\in \QMod_\mfa A$, where $\mathcal{R}$ (resp. $\mathcal{A}$) is the corresponding object of the right module $R_B\in \Mod B$ (resp. $A_A\in\Mod A$) in the quotient category. Then we have isomorphisms:
\begin{equation}\label{eq-isoext}
  \Ext_{\QMod_\mfa A}^i(\mathcal{A},\mathcal{A})\cong\Ext_{\QMod_\mfb B}^i(\mathcal{R},\mathcal{R})
\end{equation}
for all $i\ge0$. By Proposition \ref{prop-qext}, $\Ext_{\QMod_\mfb B}^i(\mathcal{R},\mathcal{R})\cong \Ext_{\QMod_\mfr R}^i(\mathcal{R},\mathcal{R})^G$ for all $i\ge0$. Hence we have
\begin{equation}\label{eq-ext}
  \Ext_{\QMod_\mfa A}^i(\mathcal{A},\mathcal{A})\cong\Ext_{\QMod_\mfr R}^i(\mathcal{R},\mathcal{R})^G,
\end{equation} for all $i\ge0$.
Applying the functor $\Hom_R(-,R)$ to the exact sequence 
$$0\to \mfr^n\to R\to R/\mfr^n\to0,$$ we obtain the following exact sequence
$$0\to \Hom_R(R/\mfr^n,R)\to R\to \Hom_R(\mfr^n,R)\to\Ext^1_R(R/\mfr^n,R)\to0,$$ and isomorphisms
$$\Ext^i_R(\mfr^n,R)\cong \Ext^{i+1}_R(R/\mfr^n,R), \text{ for }i\ge1.$$
Taking the direct limits and applying Lemma \ref{lem-qext}, we obtain the exact sequence
\begin{equation}\label{eq-sext}
  0\to \Gamma_\mfr(R)\to R\to \Hom_{\QMod_\mfr R}(\mathcal{R},\mathcal{R})\to R^1\Gamma_\mfr(R)\to0,
\end{equation}
and isomorphisms
\begin{equation}\label{eq-iso}
  \Ext^i_{\QMod_\mfr R}(\mathcal{R},\mathcal{R})\cong R^{i+1}\Gamma_\mfr(R), \text{ for }i\ge1.
\end{equation}
If $p=\Pty(R,G)\ge2$, then by (\ref{eq-iso}), $\Ext^{i-1}_{\QMod_\mfr R}(\mathcal{R},\mathcal{R})=0$ for all $i\neq p$ and $i\ge2$. By (\ref{eq-ext}),
$\Ext^{i-1}_{\QMod_\mfa A}(\mathcal{A},\mathcal{A})=0$ for all $i\neq p$ and $i\ge2$. Since $\mfa$ has the right AR-property, the exact sequence (\ref{eq-sext}) and the isomorphisms (\ref{eq-iso}) also hold if we replace $\mfr$ by $\mfa$, and $\mathcal{R}$ by $\mathcal{A}$. Hence $R^i\Gamma_\mfa(A)=0$ for $i\ge2$ and $i\neq p$.

Since $\depth_\mfr(R)\ge2$, it follows that $\Gamma_\mfr(R)=R^1\Gamma_\mfr(R)=0$. We see that the natural morphism $R\to\Hom_{\QMod_\mfr R}(\mathcal{R},\mathcal{R})$ in (\ref{eq-sext}) is an isomorphism, which implies that the natural morphism $A\to\Hom_{\QMod_\mfa A}(\mathcal{A},\mathcal{A})$ is an isomorphism since $R=A^G$ and $\Hom_{\QMod_\mfa A}(\mathcal{A},\mathcal{A})\cong \Hom_{\QMod_\mfr R}(\mathcal{R},\mathcal{R})^G$ by (\ref{eq-ext}) and Lemma \ref{lem-gact}. Replacing $\mfr$ by $\mfa$, and $\mathcal{R}$ by $\mathcal{A}$ in the exact sequence (\ref{eq-sext}), we obtain that $R^1\Gamma_\mfa(A)=0$ and $\Gamma_\mfa(A)=0$.

Finally, if $p=\Pty(R,G)=1$, then by (\ref{eq-iso}), $\Ext^{i}_{\QMod_\mfr R}(\mathcal{R},\mathcal{R})=0$ for all $i\ge1$. Hence by (\ref{eq-ext}),
$\Ext^{i}_{\QMod_\mfa A}(\mathcal{A},\mathcal{A})=0$ for all $i\ge1$, which implies that $R^i\Gamma_\mfa(A)=0$ for all $i\ge2$. Since $\Gamma_\mfr(R)=0$, we see that the natural morphism $R\to\Hom_{\QMod_\mfr R}(\mathcal{R},\mathcal{R})$ in (\ref{eq-sext}) is a monomorphism. This implies that the natural morphism $A\to\Hom_{\QMod_\mfa A}(\mathcal{A},\mathcal{A})$ is also a monomorphism. Hence $\Gamma_\mfa(A)=0$.

If $p=\Pty(R,G)=0$, similar to the above case, we have $\Ext^{i}_{\QMod_\mfa A}(\mathcal{A},\mathcal{A})=0$ for all $i\ge1$. Since $R^1\Gamma_\mfr(R)=0$, we see that the natural morphism $R\to\Hom_{\QMod_\mfr R}(\mathcal{R},\mathcal{R})$ in (\ref{eq-sext}) is an epimorphism, implying that the natural morphism $A\to\Hom_{\QMod_\mfa A}(\mathcal{A},\mathcal{A})$ is also an epimorphism. Hence $R^1\Gamma_\mfa(A)=0$.
\end{proof}

\begin{corollary} \label{cor-cm} Let $R$ be a $G$-Cohen-Macaulay algebra. Assume that $R$ satisfies Setup \ref{set}.  If $R$ has finite global dimension, then we have
$$R^i\Gamma_\mfa(M)\cong \Tor_{p-i}^A(M,D)$$ for $M\in\Mod A$ and $i\ge0$, where $p=\Pty(R,G)$ and $D=R^p\Gamma_\mfa(A)$.
\end{corollary}
\begin{proof} This is a direct consequence of Theorems \ref{thm-gcm} and \ref{thm-finproinj}.
\end{proof}

\vspace{5mm}

\subsection*{Acknowledgments}
Many thanks to James Zhang for many helpful conversations.
J.-W. He was supported by NSFC
(No. 11571239, 11401001)
and Y. Zhang by an FWO grant.

\vspace{5mm}


\end{document}